\documentclass[12pt]{article}

\usepackage{amsmath,amssymb,amsthm,latexsym,amscd,stmaryrd,mathrsfs,longtable}
\newcommand{\db}[1]{(\!({#1})\!)}

\numberwithin{equation}{section}

\newcommand{\N}{{\mathbb N}}
\newcommand{\Z}{{\mathbb Z}}
\newcommand{\Q}{{\mathbb Q}}

\newcommand{\C}{{\mathbb C}}

\DeclareMathOperator{\Aut}{Aut} 
\DeclareMathOperator{\Hom}{Hom} 
\DeclareMathOperator{\End}{End}

\DeclareMathOperator{\Res}{Res} 
\DeclareMathOperator{\id}{id} 

\DeclareMathOperator{\Span}{Span}

\DeclareMathOperator{\wt}{wt}

\newcommand{\pmul}[4]{\Phi^{#1}_{#2,#3,#4}}

\newcommand{\lw}{\Delta}
\newcommand{\wa}{b}
\newcommand{\wA}{A}
\newcommand{\wB}{B}

\newcommand{\walpha}{\alpha}

\newcommand{\zhumod}[3]{A^{#1}_{{#2},{#3}}}
\newcommand{\zhualg}[2]{A^{#1}_{{#2}}}
\newcommand{\zhuOzero}[3]{O^{{#1},0}_{{#2},{#3}}}
\newcommand{\zhuOi}[3]{O^{{#1},1}_{{#2},{#3}}}
\newcommand{\zhuOii}[3]{O^{#1,2}_{#2,#3}}
\newcommand{\zhuOiii}[3]{O^{#1,3}_{#2,#3}}
\newcommand{\zhuO}[3]{O^{#1}_{#2,#3}}
\newcommand{\zhuOit}[4]{O^{(#1;#2),1}_{#3,#4}}
\newcommand{\fo}[4]{{\Psi}^{(#1;#2)}_{#3,#4}}
\newcommand{\unitmu}[4]{E^{(#1;#2)}_{#3,#4}}
\newcommand{\hunitmu}[4]{\hat{E}^{(#1;#2)}_{#3,#4}}
\newcommand{\rbk}[2]{r(#1,#2)}
\newcommand{\glb}[2]{F(#1 | #2)}
\newcommand{\ty}[4]{Y^{(#1)}_{#2}(#3 | #4)}
\newcommand{\cty}[4]{Y^{(#1)}_{#2}(#3 ; #4)}

\newcommand{\none}{{m}}
\newcommand{\ntwo}{p}
\newcommand{\nthr}{n}
\newcommand{\ws}{s}
\renewcommand{\wr}{r}
\newcommand{\wz}{z}
\newcommand{\wx}{x}
\newcommand{\wy}{y}

\newtheorem{lemma}{Lemma}[section]
\newtheorem{theorem}[lemma]{Theorem}
\newtheorem{proposition}[lemma]{Proposition}
\newtheorem{corollary}[lemma]{Corollary}
\theoremstyle{definition}
\newtheorem{definition}[lemma]{Definition}
\newtheorem{example}[lemma]{Example}
\newtheorem{remark}[lemma]{Remark}
\title{A generalization of twisted modules over vertex algebras}
\author{Kenichiro Tanabe\footnote{This research was partially supported by JSPS Grant-in-Aid for Scientific Research No. 24540003.}
\\\\
Department of Mathematics\\
Hokkaido University\\
Kita 10, Nishi 8, Kita-Ku, Sapporo, Hokkaido, 060-0810\\
Japan\\\\
ktanabe@math.sci.hokudai.ac.jp}
\date{}
\begin{document}
\maketitle

\abstract{
For an arbitrary positive integer $T$ we introduce a notion of a $(V,T)$-module over a vertex algebra $V$,
which is a generalization of a twisted $V$-module.
Under some conditions on  $V$,
we construct an associative algebra $\zhualg{T}{m}(V)$ for $m\in(1/T)\N$ and an
$\zhualg{T}{m}(V)$-$\zhualg{T}{n}(V)$-bimodule $\zhumod{T}{n}{m}(V)$ for $n,m\in(1/T)\N$
and we establish a 
one-to-one correspondence between the set of isomorphism classes of simple left
$\zhualg{T}{0}(V)$-modules and 
that of simple $(1/T)\N$-graded $(V,T)$-modules.
}



\section{Introduction}
Twisted modules (or twisted sectors) were introduced in the study of the so-called orbifold models of conformal field theory 
(cf. \cite{DVVV,DHVW}). Let $V$ be a vertex operator algebra and $G$ a finite automorphism group of $V$.
In terms of vertex operator algebras,
the study of the orbifold models corresponds to the study of the subalgebra $V^{G}$ of $G$-invariants in $V$.
One of the main problems about $V^G$ is 
to describe the $V^G$-modules in terms of $V$ and $G$.
Twisted modules have been studied systematically as representations of $V$ related to this problem
(cf. \cite{DLM1,FLM,Le,Li}).
For $g\in G$, every $g$-twisted $V$-module becomes a $V^G$-module.
Moreover, it is conjectured that under some conditions on $V$,
every simple $V^G$-module is contained in some simple $g$-twisted 
$V$-module for some $g\in G$ (cf. \cite{DVVV}).
However, the following easy observation tells us an inconvenience of twisted $V$-modules
from the representation theoretic viewpoint:
let $g,h$ be two different elements of $G$,
$M$ a $g$-twisted $V$-module and $N$ an $h$-twisted $V$-module.
Although the direct sum $M\oplus N$ is a $V^G$-module, this is not a (twisted) $V$-module in general.
This is one of obstructions to develop the representation theory of $V^G$.

In this paper, for a vertex algebra $V$ and a positive integer $T$
we first introduce a notion of a {\it $(V,T)$-module} (cf. Definition \ref{definition:def-vt}),
which is a generalization of a twisted $V$-module, in order to resolve the inconvenience just mentioned above.
Roughly speaking, a $(V,T)$-module is a \lq\lq twisted $V$-module\rq\rq\ without automorphisms.
We next generalize some results by Zhu\cite{Z} to $(V,T)$-modules.
In \cite{Z}, if $V$ is a vertex operator algebra, then Zhu constructed an associative algebra 
$A(V)$ and established a one-to-one correspondence between the set of isomorphism classes of
the simple $A(V)$-modules and that of the simple $V$-modules with some conditions.
Some generalizations of $A(V)$ have been obtained in
\cite{DJ1,DJ2,DLM1,DLM2,DLM3} and they have played an important role in the 
representation theory of $V$.
We shall show the following results
for a vertex algebra $V$ with a grading $V=\oplus_{i=\lw}^{\infty}V_i$ such that $\lw\in\Z_{\leq 0}$,
${\bf 1}\in V_0$ and for all homogeneous element $a\in V$,
$a_{i}V_j\subset V_{\wt a-1-i+j}$,
where $V_i=0$ for $i<\Delta$. 
For every positive integer $T$ and $n,m\in (1/T)\N$,
we shall construct an associative algebra 
$\zhualg{T}{m}(V)$ and an $\zhualg{T}{n}(V)$-$\zhualg{T}{m}(V)$-bimodule $\zhumod{T}{n}{m}(V)$ 
in Theorem \ref{theorem:zhu-algebra}.
If $T=1$, then $\zhumod{T}{n}{m}(V)$ is the same as $A_{n,m}(V)$ in \cite{DJ2} and
$\zhualg{T}{n}(V)$ is the same as $A_{n}(V)$ in \cite{DLM2}.
In particular, $\zhualg{1}{0}(V)$ is the same as $A(V)$ in \cite{Z}.
For an automorphism $g$ of $V$ of finite order,
$A_{g,n,m}(V)$ in \cite{DLM1,DLM3} is a quotient of $\zhumod{|g|}{n}{m}(V)$.
For $m\in(1/T)\N$ and a left $\zhualg{T}{m}(V)$-module $U$,
we shall show in Theorem \ref{theorem:zhu-correspondence}
that the $(1/T)\N$-graded vector space 
$M(U)=\oplus_{n\in (1/T)\N}\zhumod{T}{n}{m}(V)\otimes_{\zhualg{T}{m}(V)}U$
has a structure of $(V,T)$-modules with a universal property.
In Corollary \ref{corollary:zhu-correspondence}, we establish a 
one-to-one correspondence between the set of isomorphism classes of simple
$\zhualg{T}{0}(V)$-modules and 
that of simple $(1/T)\N$-graded $(V,T)$-modules.

The organization of the paper is as follows. In Section \ref{section:vt} we introduce a notion of 
a $(V,T)$-module.
In Section \ref{section:subspace} we introduce
a subspace 
$\zhuOi{T}{n}{m}(\alpha,\beta;z)$ of $\C[\wz,\wz^{-1}]$ for $n,m\in(1/T)\N$ and $\alpha,\beta\in\Z$
and study its properties.
In Section \ref{section:zhu} we construct an associative algebra $\zhualg{T}{m}(V)$ 
and an $\zhualg{T}{n}(V)$-$\zhualg{T}{m}(V)$-bimodule $\zhualg{T}{m}(V)$ for $n,m\in (1/T)\N$ by using the results
in Section \ref{section:zhu}.
In Section \ref{section:graded-vt} we introduce a notion of a $(1/T)\N$-graded $(V,T)$-module
and study a relation between the left $\zhualg{T}{m}(V)$-modules 
and the $(1/T)\N$-graded $(V,T)$-modules.
Section 6 consists of two subsections. In Subsection \ref{subsection:determinant} we compute the determinant of a matrix used 
in Section \ref{section:subspace}. In Subsection \ref{subsection:improvement} we improve some results in \cite{MT}.
 In Section 7 we list some notations.
\section{$(V,T)$-modules
\label{section:vt}
}
We assume that the reader is familiar with the basic knowledge on
vertex algebras as presented in \cite{B,DLM1,LL}. 

Throughout this paper, $\N$ denotes the set of all non-negative integers,
$T$ is a fixed positive integer and
$(V,Y,{\mathbf 1})$ is a vertex algebra.
Recall that $V$ is the underlying vector space, 
$Y(-,\wx)$ is the linear map from $V$ to $(\End V)[[\wx,\wx^{-1}]]$,
and ${\mathbf 1}$ is the vacuum vector.
For $i,j\in\Z$, define
\begin{align*}
\Z_{\leq i}&=\{k\in\Z\ |\ k\leq i\},\\
\Z_{\geq i}&=\{k\in\Z\ |\ k\geq i\},\\
\C[\wx,\wx^{-1}]_{\leq i}&=\Span_{\C}\{\wx^k\ |\ k\leq i\},\\
\C[\wx,\wx^{-1}]_{j,i}&=\Span_{\C}\{\wx^k\ |\ j\leq k\leq i\}.
\end{align*}
For $f(\wz)\in\C[\wz,\wz^{-1}]$ and $a,b\in V$,
$f(\wz)|_{\wz^j=a_jb}$ denotes the element of $V$ obtained from $f(\wz)$ by replacing 
$\wz^j$ by $a_{j}b$ for all $j\in\Z$.
For $i,j\in\Q$, define
\begin{align}
\label{eq:delta-function}
\delta(i\leq j)&=
\left\{
\begin{array}{ll}
1& \mbox{if $i\leq j$},\\
0& \mbox{if $i>j$}.
\end{array}\right.
\end{align}

Let $M$ be a vector space over $\C$.
Define three linear injective maps
\begin{align*}
\iota_{\wx,y} : & M[[\wx^{1/T},y^{1/T}]][\wx^{-1/T},y^{-1/T},(\wx-y)^{-1}]\rightarrow M\db{\wx^{1/T}}\db{y^{1/T}},\\
\iota_{y,\wx} : &M[[\wx^{1/T},y^{1/T}]][\wx^{-1/T},y^{-1/T},(\wx-y)^{-1}]\rightarrow M\db{y^{1/T}}\db{\wx^{1/T}},\\
\iota_{\wx,\wx-y} : & M[[\wx^{1/T},y^{1/T}]][\wx^{-1/T},y^{-1/T},(\wx-y)^{-1}]\rightarrow M\db{y^{1/T}}\db{\wx-y}
\end{align*}
by 
\begin{align*}
\iota_{\wx,y}
f&=\sum_{j,k.l}a_{j,k,l}\sum_{i=0}^{\infty}\binom{l}{i}(-1)^{i}\wx^{j+l-i}y^{k+i},\\
\iota_{y,\wx}
f&=\sum_{j,k.l}a_{j,k,l}\sum_{i=0}^{\infty}\binom{l}{i}(-1)^{l-i}y^{k+l-i}\wx^{j+i},\\
\iota_{y,\wx-y}
f&=\sum_{j,k.l}a_{j,k,l}\sum_{i=0}^{\infty}\binom{j}{i}y^{k+j-i}(\wx-y)^{l+i}
\end{align*}
for $f=\sum_{j,k.l}a_{j,k,l}
\wx^jy^k(\wx-y)^l\in M[[\wx^{1/T},y^{1/T}]][\wx^{-1/T},y^{-1/T},(\wx-y)^{-1}],
a_{j,k,l}\in M$.
We can also define the map 
\begin{align*}
\iota_{x-y,y} : 
M[[\wx^{1/T},y^{1/T}]][\wx^{-1/T},y^{-1/T},(\wx-y)^{-1}]\rightarrow M\db{(\wx-y)^{1/T}}\db{y^{1/T}}
\end{align*}
similarly.
Since $\iota_{x,y}(x-y)^i=\sum_{j=0}^{\infty}\binom{i}{j}x^{i-j}(-1)^jy^{j}$
and $\iota_{x-y,y}x^i=\sum_{j=0}^{\infty}\binom{i}{j}(x-y)^{i-j}y^{j}$,
we identify $M\db{(\wx-y)^{1/T}}\db{y^{1/T}}$ with $M\db{\wx^{1/T}}\db{y^{1/T}}$ and 
$\iota_{x-y,y}$ with $\iota_{x,y}$.

Now we introduce a generalization of a twisted $V$-module.
\begin{definition}\label{definition:def-vt}
Let $M$ be a vector space over $\C$
and $Y_{M}(-,\wx)$ a linear map from $V$ to $(\End_{\C}M)[[\wx^{1/T},\wx^{-1/T}]]$.
We call $(M,Y_M)$ a {\it $(V,T)$-module}
if 
\begin{enumerate}
\item For $a\in V$ and $w\in M$, $Y_{M}(a,\wx)w\in M\db{\wx^{1/T}}$.
\item $Y_{M}({\bf 1},\wx)=\id_{M}$.
\item For $a,b\in V$ and $w\in M$,
there is $\glb{a,b,w}{\wx,\wy}\in M[[\wx^{1/T},y^{1/T}]][\wx^{-1/T},y^{-1/T},(\wx-y)^{-1}]$
such that 
\begin{align*}
\iota_{\wx,y}\glb{a,b,w}{\wx,\wy}&=Y_{M}(a,\wx)Y_{M}(b,y)w,\\
\iota_{y,\wx}\glb{a,b,w}{\wx,\wy}&=Y_{M}(b,y)Y_{M}(a,\wx)w,\quad\mbox{and }\\
\iota_{y,\wx-y}\glb{a,b,w}{\wx,\wy}&=Y_{M}(Y(a,\wx-y)b,y)w.
\end{align*}
\end{enumerate}
\end{definition}

We note that in Definition \ref{definition:def-vt},
$\glb{a,b,w}{x,y}$ is uniquely determined by $a,b\in V$ and $w\in M$
since $\iota_{\wx,\wy}$ is an injection. 
For a $(V,T)$-module $M$, 
a subspace $N$ of $M$ is called {\it $(V,T)$-submodule} of $M$
if $(N,Y_M|_{N})$ is a $(V,T)$-module,
where $Y_M|_{N}$ is the restriction of $Y_{M}$ to $N$.
A non-zero $(V,T)$-module $M$ is called {\it simple} 
if there is no submodule of $M$ except $0$ and $M$ itself.
For a submodule $N$ of a $(V,T)$-module $M$,
the quotient space $M/N$ is clearly a $(V,T)$-module.
For a set of $(V,T)$-modules $\{M_{i}\}_{i\in I}$,
the direct sum $\oplus_{i\in I}M_i$ is  a $(V,T)$-module.

\begin{remark}\label{remark:twisted-vt}
It follows from Lemma \ref{lemma:Borcherds} below that
every $(V,1)$-module is a $V$-module and vice versa
and that 
every $g$-twisted $V$-module is a $(V,|g|)$-module for an automorphism $g$ of $V$ of finite order.

Let $T^{\prime}$ be a positive multiple of $T$.
Then every $(V,T)$-module is a $(V,T^{\prime})$-module.
Thus, for positive integers $T_1$ and $T_2$
the direct sum of a $(V,T_1)$-module and a $(V,T_2)$-module
becomes a $(V,T_3)$-module, where $T_3$ is a positive common multiple of $T_1$ and $T_2$.
Thus, $(V,T)$-modules are closed under direct sums in this sense,
while twisted $V$-modules are not as stated in the introduction. 
\end{remark}

\begin{example}\label{example:vs3}
We introduce an easy example of simple $(V,T)$-modules which is not a twisted $V$-module. 
Let $U$ be a simple vertex operator algebra.
Suppose the symmetric group $S_3$ of degree $3$ is an automorphism group of $U$.
Let $\sigma,\tau \in S_3$ such that $|\sigma|=3$ and $|\tau|=2$ and
$M=\oplus_{j\in(1/3)\N}M(j)$ a simple $\sigma$-twisted $U$-module \cite{DLM1}.
It follows from Remark \ref{remark:twisted-vt} that
$M$ is a $(U,3)$-module. Restricting $Y_{M}$ to $U^{\langle \tau \rangle}$,
$M$ becomes a $(U^{\langle \tau \rangle},3)$-module.
We shall show $M$ is a simple $(U^{\langle \tau \rangle},3)$-module.
Let $W$ be a non-zero $(U^{\langle \tau \rangle},3)$-submodule of $M$.
We denote the subspace $\oplus_{j\in i/3+\N}M(j)$ of $M$ by $M^{i}, i=0,1,2$.
Since $\tau\sigma\tau=\sigma^{-1}\neq \sigma$,  an improvement of 
\cite[Theorem 2]{MT} (see Subsection \ref{subsection:improvement})
implies that $M^0,M^1$ and $M^2$ are all inequivalent simple $U^{S_3}$-modules.
Thus, $W$ contains at least one of $M^0,M^1$ and $M^2$ since $U^{S_3}\subset U^{\langle\tau\rangle}$.  
We denote the eigenspace $\{u\in U\ |\ \sigma u=e^{-2\pi\sqrt{-1}r/3}u\}$
of $\sigma$ by $U^{(\sigma,r)}$,$r=0,1,2$.
It follows by \cite[Proposition 3.3]{DM} and \cite[Theorem 1]{HMT} that $U^{\langle \tau \rangle}\not\subset
U^{\langle \sigma \rangle}$ and hence there exists $a=a^0+a^1+a^2\in U^{\langle \tau \rangle}$,
$a^r\in U^{(\sigma,r)}$ such that at least one of $a^1,a^2$ is not zero.
Since 
\begin{align*}
Y_{M}(a,x)&=\sum_{i\in \Z}a^0_ix^{-i-1}+\sum_{i\in 1/3+\Z}a^1_ix^{-i-1}+\sum_{i\in 2/3+\Z}a^2_ix^{-i-1}
\end{align*}
and $M$ is a simple $\sigma$-twisted $U$-module,
$W$ contains at least two of $M^{0},M^{1}$ and $M^{2}$.
Repeating the same argument,
we obtain that 
$M$ is a simple $(U^{\langle \tau \rangle},3)$-module.

Since at least one of $a^1,a^2$ above is not zero, $M$ is not a $U^{\langle \tau \rangle}$-module.
Suppose $M$ is a $g$-twisted $U^{\langle \tau \rangle}$-module for some $g\in \Aut U^{\langle \tau \rangle}$
of order $3$.
Then, the eigenspace
$(U^{\langle \tau \rangle})^{(g,r)}=\{v\in U^{\langle \tau \rangle}\ |\ gv=e^{-2\pi\sqrt{-1}r/3}v\}$
of $g$
is a subspace of $U^{(\sigma,r)}$ for each $r=0,1,2$  since
$Y_{M}(b,x)=\sum_{j\in r/3+\Z}b_jx^{-j-1}$ for $b\in (U^{\langle \tau \rangle})^{(g,r)}$.
Therefore, $(U^{\langle \tau \rangle})^{(g,1)}=(U^{\langle \tau \rangle})^{(g,2)}=0$
since there is no representation $\rho$ of $S^3$ such that $\rho(\sigma)=e^{-2\pi\sqrt{-1}r/3}$
and $\rho(\tau)=1$ for $r=1,2$.
This contradicts to that the order of $g$ is equal to $3$.
We conclude that $M$ is not a twisted $U^{\langle \tau \rangle}$-module.
\end{example}

Let $M$ be a vector space.
For $s=0,\ldots,T-1$ and $X(\wx,y)=\sum_{i,j\in(1/T)\Z}X_{ij}x^iy^j
\in M[[\wx^{1/T},\wx^{-1/T},y^{1/T},y^{-1/T}]]$,
$X_{ij}\in M$, we define 
\begin{align}
\label{eq:xxyrx}
X(x,y)^{s,x}&=\sum_{\begin{subarray}{c}i\in s/T+\Z\\j\in(1/T)\Z\end{subarray}}
X_{ij}\wx^iy^j\quad\mbox{ and }\nonumber\\
X(x,y)^{s,y}&=\sum_{\begin{subarray}{c}i\in (1/T)\Z\\j\in s/T+\Z\end{subarray}}
X_{ij}x^iy^j
\end{align}
in $M[[x^{1/T},\wx^{-1/T},y^{1/T},y^{-1/T}]]$.
In the same way,
for $s=0,\ldots,T-1$ and $X(x,y)=\sum_{i,j\in(1/T)\Z}\sum_{k\in\Z}
X_{ijk}x^iy^j(x-y)^{k}\in M[[x^{1/T},y^{1/T}]][x^{-1/T},y^{-1/T},(x-y)^{-1}]$,
we define 
\begin{align}
\label{eq:xxyrx-2}
X(\wx,y)^{s,\wx}&=\sum_{\begin{subarray}{c}i\in s/T+\Z\\j\in(1/T)\Z\end{subarray}}
\sum_{k\in\Z}
X_{ijk}\wx^iy^j(x-y)^{k}\quad\mbox{ and }\nonumber\\
X(\wx,y)^{s,y}&=\sum_{\begin{subarray}{c}i\in (1/T)\Z\\j\in s/T+\Z\end{subarray}}
\sum_{k\in\Z}
X_{ijk}\wx^iy^j(x-y)^{k}
\end{align}
in $M[[x^{1/T},y^{1/T}]][x^{-1/T},y^{-1/T},(x-y)^{-1}]$.
Clearly
\begin{align*}
\sum_{s=0}^{T-1}X(\wx,y)^{s,x}&=\sum_{s=0}^{T-1}X(\wx,y)^{s,y}=X(\wx,y).
\end{align*}
For $0\leq s\leq T-1$, $j\in s/T+\Z, k\in(1/T)\Z$ and $l\in\Z$, 
the following fact is well known and straightforward:
\begin{align}
\label{eq:uniform}
&\wx_0^{-1}\delta(\dfrac{\wx_1-\wx_2}{\wx_0})\iota_{\wx_1,\wx_2}\big((\wx_1^{j}\wx_2^{k}\wx_0^{l})|_{\wx_0=\wx_1-\wx_2}\big)\nonumber\\
&\quad-\wx_0^{-1}\delta(\dfrac{-\wx_2+\wx_1}{\wx_0})\iota_{\wx_2,\wx_1}\big((\wx_1^{j}\wx_2^{k}\wx_0^{l})|_{\wx_0=\wx_1-\wx_2}\big)\nonumber\\
&=\wx_1^{-1}(\dfrac{\wx_2+\wx_0}{\wx_1})^{-s/T}\delta(
\dfrac{\wx_2+\wx_0}{\wx_1})
\iota_{\wx_2,\wx_0}\big((\wx_1^{j}\wx_2^{k}\wx_0^{l})|_{\wx_1=\wx_2+\wx_0}\big).
\end{align}

The argument in the proof of the following lemma is well known (cf. \cite[Sections 3.2--3.4]{LL}).
\begin{lemma}\label{lemma:Borcherds}
Let $A(\wx_1,\wx_2)\in M\db{\wx_1^{1/T}}\db{\wx_2^{1/T}},B(\wx_2,\wx_1)\in M\db{\wx_2^{1/T}}\db{\wx_1^{1/T}}$,and
$C(\wx_2,\wx_0)\in M\db{\wx_2^{1/T}}\db{\wx_0}$.
Then, the three following conditions are equivalent.
\begin{enumerate}
\item There is $F\in M[[\wx_1^{1/T},\wx_2^{1/T}]][\wx_1^{-1/T},\wx_2^{-1/T},(\wx_1-\wx_2)^{-1}]$
such that
\begin{align*}
\iota_{\wx_1,\wx_2}F&=A(\wx_1,\wx_2),&
\iota_{\wx_2,\wx_1}F&=B(\wx_2,\wx_1),\mbox{ and }\\
\iota_{\wx_2,\wx_1-\wx_2}F&=C(\wx_2,\wx_1-\wx_2).
\end{align*}
\item
There are $C^{[s]}(\wx_2,\wx_0)\in M\db{\wx_2^{1/T}}\db{\wx_0}, s=0,\ldots,T-1$ such that
$\sum_{s=0}^{T-1}C^{[s]}(\wx_2,\wx_0)=C(\wx_2,\wx_0)$ and 
\begin{align}
\label{eq:s-Borcherds}
&\wx_0^{-1}\delta(\dfrac{\wx_1-\wx_2}{\wx_0})A(\wx_1,\wx_2)^{s,\wx_1}
-\wx_0^{-1}\delta(\dfrac{-\wx_2+\wx_1}{\wx_0})B(\wx_2,\wx_1)^{s,\wx_1}\nonumber\\
&=\wx_1^{-1}(\dfrac{\wx_2+\wx_0}{\wx_1})^{-s/T}\delta(\dfrac{\wx_2+\wx_0}{\wx_1})
C^{[s]}(\wx_2,\wx_0).
\end{align}
\item
There are positive integers $l,q$ and 
$C^{[s]}(\wx_2,\wx_0)\in M\db{\wx_2^{1/T}}\db{\wx_0}, s=0,\ldots,T-1$ such that
$\sum_{s=0}^{T-1}C^{[s]}(\wx_2,\wx_0)=C(\wx_2,\wx_0)$,
\begin{align}
\label{eq:Borcherds-commutative}
(\wx_1-\wx_2)^{l}A(\wx_1,\wx_2)&=(\wx_1-\wx_2)^{l}B(\wx_2,\wx_1)
\end{align}
in $M[[x_1^{1/T},x_2^{1/T},x_1^{-1/T},x_2^{-1/T}]]$
and
\begin{align}
\label{eq:Borcherds-associative}
&\iota_{\wx_0,\wx_2}(\wx_0+\wx_2)^{-s/T+q}(A(\wx_1,\wx_2)^{s,\wx_1})|_{\wx_1=\wx_0+\wx_2}\nonumber\\
&=\iota_{\wx_2,\wx_0}(\wx_2+\wx_0)^{-s/T+q}C^{[s]}(\wx_2,\wx_0)
\end{align}
in $M[[x_0,x_2^{1/T},x_0^{-1},x_2^{-1/T}]]$.
\end{enumerate}
In this case, $F$ and $C^{[s]}(\wx_2,\wx_0),\ s=0,\ldots,T-1$ are uniquely determined by 
$A(\wx_1,\wx_2),B(\wx_2,\wx_1)$ and $C(\wx_2,\wx_0)$.
\end{lemma}
\begin{proof}
We show (1) implies (2).
Define $C^{[s]}(\wx_2,\wx_0)\in M\db{x_2^{1/T}}\db{x_0}$ by
$C^{[s]}(\wx_2,\wx_1-\wx_2)=\iota_{\wx_2,\wx_1-\wx_2}F^{s,\wx_1}
\in M\db{x_2^{1/T}}\db{x_1-x_2}$ for $s=0,\ldots,T-1$.
Clearly,  $\sum_{s=0}^{T-1}C^{[s]}(\wx_2,\wx_0)=C(\wx_2,\wx_0)$.
Since  
$\iota_{\wx_1,\wx_2}F^{s,\wx_1}=A(\wx_1,\wx_2)^{s,\wx_1}$ and
$\iota_{\wx_2,\wx_1}F^{s,\wx_1}=B(\wx_2,\wx_1)^{s,\wx_1}$ for $s=0,\ldots,T-1$,
\eqref{eq:s-Borcherds} follows from \eqref{eq:uniform}.

We show (2) implies (3).
Let $l$ be a positive integer such that 
$x_0^{l}C^{[s]}(x_2,x_0)\in M\db{x_2^{1/T}}[[x_0]]$ for all $s=0,\ldots,T-1$. 
Multiplying \eqref{eq:s-Borcherds} by $x_0^{l}$ and then
taking $\Res_{x_0}$, we have 
$(\wx_1-\wx_2)^{l}A(\wx_1,\wx_2)^{\ws,\wx_1}=(\wx_1-\wx_2)^{l}B(\wx_2,\wx_1)^{\ws,\wx_1}$ and hence
\eqref{eq:Borcherds-commutative}.
Let $q$ be a positive integer such that
$x_1^{-s/T+q}B(x_2,x_1)^{s,x_1}\in M\db{x_2^{1/T}}[[x_1]]$ for all $s=0,\ldots,T-1$. 
Multiplying \eqref{eq:s-Borcherds} by $x_1^{-s/T+q}$ and then
taking $\Res_{x_1}$, we have \eqref{eq:Borcherds-associative}.

We show (3) implies (1).
Since the left-hand side of \eqref{eq:Borcherds-commutative} is an element of 
$M\db{\wx_1^{1/T}}\db{\wx_2^{1/T}}$ and the right-hand side of \eqref{eq:Borcherds-commutative}
is an element of 
$M\db{\wx_2^{1/T}}\db{\wx_1^{1/T}}$,
$G=(\wx_1-\wx_2)^{l}A(\wx_1,\wx_2)(=(\wx_1-\wx_2)^{l}B(\wx_2,\wx_1))$
is an element of $M[[\wx_1^{1/T},\wx_2^{1/T}]][\wx_1^{-1/T},\wx_2^{-1/T}]$.
Define 
\begin{align*}
F&=(\wx_1-\wx_2)^{-l}G\in 
M[[\wx_1^{1/T},\wx_2^{1/T}]][\wx_1^{-1/T},\wx_2^{-1/T},(\wx_1-\wx_2)^{-1}].
\end{align*}
It is clear that $\iota_{\wx_1,\wx_2}F=A(\wx_1,\wx_2)$ and $\iota_{\wx_2,\wx_1}F=B(\wx_2,\wx_1)$.
Applying the same argument to \eqref{eq:Borcherds-associative}, we obtain
$H_{s}\in M[[\wx_1^{1/T},\wx_2^{1/T}]][\wx_1^{-1/T},\wx_2^{-1/T},(\wx_1-\wx_2)^{-1}]$,
$s=0,\ldots,T-1$ such that
\begin{align*}
\iota_{\wx_1-\wx_2,\wx_2}H_{s}&=A(\wx_1,\wx_2)^{s,x_1}=A((\wx_1-x_2)+x_2,\wx_2)^{s,x_1}\mbox{ and}\\
\iota_{\wx_2,\wx_1-\wx_2}H_{s}&=C^{[s]}(\wx_2,\wx_1-\wx_2).
\end{align*}
Since $M\db{(x_1-x_2)^{1/T}}\db{x_2^{1/T}}=M\db{x_1^{1/T}}\db{x_2^{1/T}}$ and $\iota_{\wx_1,\wx_2}$
is injective, we have $F^{s,\wx_1}=H_s$ for all $s=0,\ldots,T-1$
and therefore $\iota_{\wx_2,\wx_1-\wx_2}F=C(\wx_2,\wx_1-\wx_2)$.

We show $F$ and $C^{[s]}(\wx_2,\wx_0),\ s=0,\ldots,T-1$ are uniquely determined.
Since $\iota_{\wx_1,\wx_2}$ is injective and $\iota_{\wx_1,\wx_2}F=A(\wx_1,\wx_2)$, 
$F$ is uniquely determined.
In the above argument that (3) implies (1),
we have constructed  $F$
such that $\iota_{\wx_1,\wx_2}F=A(\wx_1,\wx_2)$ and
$\iota_{x_2,x_1-x_2}F^{s,x_1}=C^{[s]}(x_2,x_1-x_2)$.
Thus, $C^{[s]}(\wx_2,\wx_0),\ s=0,\ldots,T-1$ in (3) are uniquely determined.
A similar argument shows that  $C^{[s]}(\wx_2,\wx_0),\ s=0,\ldots,T-1$ in (2) are uniquely determined
and that $C^{[s]}(\wx_2,\wx_0)$ in (2) is the same as that in (3) for each $s$.
\end{proof}

\begin{remark}\label{remark:Borcherds-natural}
The following facts for \eqref{eq:s-Borcherds} are well known and straightforward.
\begin{enumerate}
\item
A direct computation shows that \eqref{eq:s-Borcherds} is 
equivalent to 
\begin{align}
\label{eq:general-Borcherds}
&\wx_0^{-1}\delta(\dfrac{\wx_1-\wx_2}{\wx_0})A(\wx_1,\wx_2)
-\wx_0^{-1}\delta(\dfrac{-\wx_2+\wx_1}{\wx_0})B(\wx_2,\wx_1)\nonumber\\
&=\dfrac{1}{T}\sum\limits_{j=0}^{T-1}\wx_1^{-1}\delta(e^{2\pi\sqrt{-1}j/T}
\dfrac{(\wx_2+\wx_0)^{1/T}}{\wx_1^{1/T}})\sum\limits_{s=0}^{T-1}
e^{2\pi\sqrt{-1}js/T}C^{[s]}(\wx_2,\wx_0).
\end{align}
\item
If we write 
$A(x_1,x_2)=\sum_{p,q}A_{(p,q)}x_1^{-p-1}x_2^{-q-1}$,
$B(x_2,x_1)=\sum_{p,q}B_{(p,q)}x_2^{-p-1}x_1^{-q-1}$ and
$C^{[s]}(x_2,x_0)=\sum_{p,q}C^{[s]}_{(p,q)}x_2^{-p-1}x_0^{-q-1}$, where
$A_{(p,q)},B_{(p,q)},C_{(p,q)}\in M$, then we have
\begin{align}
\label{eq:ABC-Coeff}
&\sum_{i=0}^{\infty}\binom{l}{i}(-1)^{i}\big(A_{(l+j-i,k+i)}+(-1)^{l+1}B_{(l+k-i,j+i)}\big)\nonumber\\
&=\sum_{i=0}^{\infty}\binom{j}{i}C^{[-s]}_{(j+k-i,l+i)}
\end{align}
for $0\leq s\leq T-1$, $j\in \ws/T+\Z, k\in (1/T)\Z$ and $l\in\Z$
by comparing the coefficients of both sides of \eqref{eq:s-Borcherds}.
Thus, a direct computation shows that \eqref{eq:s-Borcherds} is also 
equivalent to the condition that
\begin{align}
\label{eq:residue-Borcherds}
&\Res_{x_1}A(x_1,x_2)\iota_{x_1,x_2}\big(x_1^{j}x_2^{k}(x_1-x_2)^l\big)\nonumber\\
&\quad{}-\Res_{x_1}B(x_2,x_1)\iota_{x_2,x_1}\big(x_1^{j}x_2^{k}(x_1-x_2)^l\big)\nonumber\\
&=\Res_{x_1-x_2}C^{[-s]}(x_2,x_1-x_2)\iota_{x_2,x_1-x_2}\big(x_1^{j}x_2^{k}(x_1-x_2)^l\big)
\end{align}
in $M[[x_2^{1/T},x_2^{-1/T}]]$ for all $0\leq s\leq T-1$, $j\in \ws/T+\Z, k\in (1/T)\Z$ and $l\in\Z$.
Here, $\Res_{x}$ is defined by
\begin{align*}
\Res_{x}f(x)=f_{-1}
\end{align*}
for $f(x)=\sum_{i\in(1/T)\Z}f_ix^i\in M[[x^{1/T},x^{-1/T}]]$.
\end{enumerate}
\end{remark}

\begin{remark}\label{remark:bound}
For $q\in\Z$ we denote by $M\db{\wx_2^{1/T}}\db{\wx_0}_{\geq q}$
the set of all elements in $M\db{\wx_2^{1/T}}\db{\wx_0}$ of the form
$\sum_{\begin{subarray}{c}i\in(1/T)\Z\\j\in\Z_{\geq q}\end{subarray}}
X_{ij}\wx_2^{i}\wx_0^{j}$.
Suppose
$C(\wx_2,\wx_0)$ in Lemma \ref{lemma:Borcherds} is an element of $M\db{\wx_2^{1/T}}\db{\wx_0}_{\geq q}$.
Since $\iota_{\wx_2,\wx_1-\wx_2}\wx_1^{j}\wx_2^{k}(\wx_1-\wx_2)^{l}=
\sum_{i=0}^{\infty}\binom{j}{i}\wx_2^{k+j-i}(\wx_1-\wx_2)^{l+i}$,
we see that 
$F$ in Lemma \ref{lemma:Borcherds} (1) has the form $F=(\wx_1-\wx_2)^{q}G$, where $G\in M[[\wx_1^{1/T},\wx_2^{1/T}]][\wx_1^{-1/T},\wx_2^{-1/T}]$.
Thus, $C^{[s]}(\wx_2,\wx_1-\wx_2)=\iota_{\wx_2,\wx_1-\wx_2}F^{s,\wx_1}\in M\db{\wx_2^{1/T}}\db{\wx_1-\wx_2}_{\geq q}$
and hence $C^{[s]}(\wx_2,\wx_0)\in M\db{\wx_2^{1/T}}\db{\wx_0}_{\geq q}$
for all $s=0,\ldots,T-1$.
\end{remark}

Let $M$ be a $(V,T)$-module.
For $a\in V$ and $s=0,\ldots,T-1$, we define $Y^{s}_{M}(a,\wx)$ by
\begin{align}
\label{eq:Ys}
Y^{s}_{M}(a,\wx)&=\sum_{\begin{subarray}{c}i\in s/T+\Z\end{subarray}}
a_{i}\wx^{-i-1}.
\end{align}
Let $a,b\in V$ and $w\in M$.
We apply Lemma \ref{lemma:Borcherds} to 
$A(\wx_1,\wx_2)=Y_{M}(a,\wx_1)Y_{M}(b,\wx_2)w$, $B(\wx_2,\wx_1)=Y_{M}(b,\wx_2)Y_{M}(a,\wx_1)w$ and 
$C(\wx_2,\wx_0)=Y_{M}(Y(a,\wx_0)b,\wx_2)w$. 
In this case $F$ in Lemma \ref{lemma:Borcherds} (1)
is equal to $\glb{a,b,w}{x_1,x_2}$ in Definition \ref{definition:def-vt} (3).
We denote by $\ty{-s}{M}{a,b}{x_2,x_0}(w)$
the element $C^{[s]}(\wx_2,\wx_0)$ of $M\db{\wx_2^{1/T}}\db{\wx_0}$,
$s=0,\ldots,T-1$ in this case.
That is, 
\begin{align}
\label{eq:expnad-glb}
\ty{s}{M}{a,b}{x_2,x_1-x_2}(w)=\iota_{x_2,x_1-x_2}(\glb{a,b,w}{x_1,x_2}^{-s,x_1}),
\end{align}
where $\glb{a,b,w}{x_1,x_2}^{-s,x_1}$ is defined by \eqref{eq:xxyrx-2}.
The conditions in Lemma \ref{lemma:Borcherds} (2) become
\begin{align}
\label{eq:sum-Y}
\sum_{s=0}^{T-1}\ty{s}{M}{a,b}{\wx_2,\wx_0}(w)&=Y_{M}(Y(a,\wx_0)b,\wx_2)w
\end{align}
and
\begin{align}
\label{eq:Y-abw}
&\wx_0^{-1}\delta(\dfrac{\wx_1-\wx_2}{\wx_0})Y^{s}_{M}(a,\wx_1)Y_{M}(b,\wx_2)w
-\wx_0^{-1}\delta(\dfrac{-\wx_2+\wx_1}{\wx_0})Y_{M}(b,\wx_2)Y^{s}_{M}(a,\wx_1)w\nonumber\\
&=\wx_1^{-1}(\dfrac{\wx_2+\wx_0}{\wx_1})^{s/T}\delta(
\dfrac{\wx_2+\wx_0}{\wx_1})\ty{s}{M}{a,b}{\wx_2,\wx_0}(w).
\end{align}
The uniqueness of $\glb{a,b,w}{x_1,x_2}$ for each $a,b\in V$ and $w\in M$ implies that
for fixed $a,b\in V$ the map $\ty{s}{M}{a,b}{\wx_2,\wx_0}: M\rightarrow M\db{\wx_2^{1/T}}\db{\wx_0}$ is linear
and that the map $V\times V\ni(a,b)\mapsto \ty{s}{M}{a,b}{\wx_2,\wx_0}
\in \Hom_{\C}(M, M\db{\wx_2^{1/T}}\db{\wx_0})$ is bilinear.
We write 
\begin{align*}
\ty{s}{M}{a,b}{\wx_2,\wx_0}&=\sum_{i\in(1/T)\Z}\sum_{j\in\Z}\cty{s}{M}{a,b}{i,j}\wx_2^{-i-1}\wx_0^{-j-1},
\end{align*}
where $\cty{s}{M}{a,b}{i,j}\in \End_{\C}M$.

\begin{remark}\label{remark:twisted-vt-Y}
Let $g$ be an automorphism of $V$ of finite order,
$t$ a positive multiple of $|g|$
and $(M,Y_{M})$ a $g$-twisted $V$-module.
As stated in Remark \ref{remark:twisted-vt},
$(M,Y_{M})$ is a $(V,t)$-module by Lemma \ref{lemma:Borcherds}.
We explain what is $\ty{s}{M}{a,b}{\wx_2,\wx_0}$ for $a,b\in V$ and $s=0,\ldots,t-1$
in this case.
We denote by $V^{(g,r)}, r=0,\ldots,t-1$ 
the eigenspace $\{v\in V\ |\ gv=e^{-2\pi\sqrt{-1}r/t}v\}$ of $g$.
For $a\in V$, we denote by $a^{(g,r)}$ 
the $r$-th component of $a$ in the decomposition 
$V=\oplus_{r=0}^{t-1}V^{(g,r)}$, that is,
$a=\sum_{r=0}^{t-1}a^{(g,r)}, a^{(g,r)}\in V^{(g,r)}$.

Let $0\leq \ws\leq t-1$, $a,b\in V$ and $w\in M$.
Since
\begin{align*}
(Y_{M}(a,x_1)Y_{M}(b,x_2)w)^{-s,x_1}&=Y_{M}(a^{(g,s)},x_1)Y_{M}(b,x_2)w\quad\mbox{ and}\\
(Y_{M}(b,x_2)Y_{M}(a,x_1)w)^{-s,x_1}&=Y_{M}(b,x_2)Y_{M}(a^{(g,s)},x_1)w,
\end{align*} 
it follows by \eqref{eq:s-Borcherds} that
\begin{align}
\label{eq:twiste-equal}
\ty{s}{M}{a,b}{\wx_2,\wx_0}(w)&=Y_{M}(Y(a^{(g,s)},x_0)b,x_2)w.
\end{align}
\end{remark}

Let $a,b\in V$, $w\in M$, $j,k\in(1/T)\Z, l\in\Z$ and
$s$ the integer uniquely determined by the conditions $0\leq s\leq T-1$ and $s/T\equiv j\pmod{\Z}$.
It follows
by \eqref{eq:ABC-Coeff} or by comparing the coefficients of both sides of \eqref{eq:Y-abw}
that
\begin{align}
\label{eq:Borcherds-Coeff}
&\sum_{i=0}^{\infty}\binom{j}{i}\cty{s}{M}{a,b}{j+k-i,l+i}(w)\nonumber\\
&=
\sum_{i=0}^{\infty}\binom{l}{i}(-1)^{i}(a_{l+j-i}b_{k+i}+(-1)^{l+1}b_{l+k-i}a_{j+i})w.
\end{align}
It follows by \eqref{eq:residue-Borcherds} that
\begin{align}
\label{eq:residue-yayb}
&\Res_{x_1-x_2}\iota_{x_2,x_1-x_2}\big(x_1^{j}x_2^{k}(x_1-x_2)^l\big)\ty{s}{M}{a,b}{x_2,x_1-x_2}(w)\nonumber\\
&=\Res_{x_1}\iota_{x_1,x_2}\big(x_1^{j}x_2^{k}(x_1-x_2)^l\big)Y_{M}(a,x_1)Y_{M}(b,x_2)w \nonumber\\
&\quad{}-\Res_{x_1}\iota_{x_2,x_1}\big(x_1^{j}x_2^{k}(x_1-x_2)^l\big)Y_{M}(b,x_2)Y_{M}(a,x_1)w\nonumber\\
&=\Res_{x_1}\iota_{x_1,x_2}\big(x_1^{j}x_2^{k}(x_1-x_2)^l\big)Y^{s}_{M}(a,x_1)Y_{M}(b,x_2)w \nonumber\\
&\quad{}-\Res_{x_1}\iota_{x_2,x_1}\big(x_1^{j}x_2^{k}(x_1-x_2)^l\big)Y_{M}(b,x_2)Y^{s}_{M}(a,x_1)w. 
\end{align}

\begin{lemma}\label{lemma:express-ysab}
We use the notation above.
Let $L$ be an integer such that $a_{i}b=0$ for all $i\in \Z_{\geq L+1}$.
Then
\begin{align}
\label{eq:eq-express-ysab}
&\cty{s}{M}{a,b}{j+k,l}(w)\nonumber\\
&=\sum_{m=0}^{L-l}\binom{-j}{m}
\sum_{i=0}^{\infty}\binom{l+m}{i}(-1)^{i}(a_{l+m+j-i}b_{k-m+i}+(-1)^{l+m+1}b_{l+k-i}a_{j+i})w.
\end{align}
\end{lemma}
\begin{proof}
It follows from Remark \ref{remark:bound} that
$\ty{s}{M}{a,b}{\wx_2,\wx_0}(w)\in M\db{x_2^{1/T}}\db{x_0}_{\geq -L-1}$.
Thus, if $l>L$, then the both-sides of \eqref{eq:eq-express-ysab} are equal to $0$.
Suppose $l\leq L$.
Define 
\begin{align*}
R(m)&=
\sum_{i=0}^{\infty}\binom{m}{i}(-1)^{i}(a_{m+j-i}b_{k-m+l+i}+(-1)^{m+1}b_{l+k-i}a_{j+i})w
\end{align*}
for $m\in \Z_{\leq L}$.
Since
\begin{align*}
&\begin{pmatrix}
1&0&\cdots&\cdots&0\\
\binom{j}{1}&1&\ddots&&\vdots\\
\binom{j}{2}&\ddots&\ddots&\ddots&\vdots\\
\vdots&\ddots&\ddots&\ddots&0\\
\binom{j}{L-l}&\cdots&\binom{j}{2}&\binom{j}{1}&1
\end{pmatrix}
\begin{pmatrix}
\cty{s}{M}{a,b}{j+k+l-L,L}(w)\\
\cty{s}{M}{a,b}{j+k+l-L+1,L-1}(w)\\
\vdots\\
\cty{s}{M}{a,b}{j+k,l}(w)
\end{pmatrix}\\
&=\begin{pmatrix}
R(L)\\
R(L-1)\\
\vdots\\
R(l)
\end{pmatrix}
\end{align*} 
by \eqref{eq:Borcherds-Coeff},
we have
\begin{align*}
&\begin{pmatrix}
\cty{s}{M}{a,b}{j+k+l-L,L}(w)\\
\cty{s}{M}{a,b}{j+k+l-L+1,L-1}(w)\\
\vdots\\
\cty{s}{M}{a,b}{j+k,l}(w)
\end{pmatrix}\\
&=\begin{pmatrix}
1&0&\cdots&\cdots&0\\
\binom{-j}{1}&1&\ddots&&\vdots\\
\binom{-j}{2}&\ddots&\ddots&\ddots&\vdots\\
\vdots&\ddots&\ddots&\ddots&0\\
\binom{-j}{L-l}&\cdots&\binom{-j}{2}&\binom{-j}{1}&1
\end{pmatrix}
\begin{pmatrix}
R(L)\\
R(L-1)\\
\vdots\\
R(l)
\end{pmatrix}.
\end{align*} 
This implies \eqref{eq:eq-express-ysab}.
\end{proof}

Let $a,b\in V$, $w\in M$, $j,k\in(1/T)\Z, l\in\Z$ and
$s$ the integer uniquely determined by the conditions $0\leq s\leq T-1$ and $s/T\equiv j\pmod{\Z}$.
It follows by Lemma \ref{lemma:Borcherds} that $\glb{a,{\bf 1},w}{x_1,x_2}=Y_{M}(a,x_1)w$ 
since 
\begin{align*}
Y_{M}(a,x_1)Y_{M}({\bf 1},x_2)w&=Y_{M}(a,x_1)w\\
&\in M[[x_1^{1/T},x_2^{1/T}]][x_1^{-1/T},x_2^{-1/T},(x_1-x_2)^{-1}].
\end{align*}
Comparing the coefficients of
\begin{align*}
&\iota_{x_2,x_1-x_2}x_1^{j}\ty{s}{M}{a,{\bf 1}}{\wx_2,\wx_1-\wx_2}(w)\\
&=\sum_{k\in(1/T)\Z}\sum_{l\in\Z}\sum_{i=0}^{\infty}\binom{j}{i}
\cty{s}{M}{a,{\bf 1}}{j+k-i,l+i}(w)x_2^{-k-1}(x_1-x_2)^{-l-1}
\end{align*}
and 
\begin{align*}
&\iota_{x_2,x_1-x_2}x_1^{j}Y_{M}^{s}(a,x_1)w\\
&=\sum_{
k\in\Z
}\sum_{l\in\Z}\binom{k}{-l-1}(-1)^{l+1}a_{j+k+l+1}w x_2^{-k-1}(x_1-x_2)^{-l-1},
\end{align*}
we have
\begin{align}
\label{eq:Borcherds-a-1}
&\sum\limits_{i=0}^{-l-1}\dbinom{j}{i}\cty{s}{M}{a,{\bf 1}}{j+k-i,l+i}(w)\nonumber\\
&=\left\{
\begin{array}{ll}
\dbinom{k}{-l-1}(-1)^{l+1}a_{j+k+l+1}w& \mbox{ if }k\in\Z,\\
0& \mbox{ if }k\not\in\Z.
\end{array}
\right.
\end{align}
Here, we used $\ty{s}{M}{a,{\bf 1}}{x_2,x_1-x_2}(w)\in M\db{x_2^{1/T}}[[x_1-x_2]]$ by Remark \ref{remark:bound}.
We can also obtain \eqref{eq:Borcherds-a-1} by taking $b={\bf 1}$ in \eqref{eq:Borcherds-Coeff}.
Taking $l=-1$ in \eqref{eq:Borcherds-a-1}, we have
\begin{align}
\label{eq:Borcherds-a-1-l}
\cty{s}{M}{a,{\bf 1}}{i,-1}(w)&=\left\{
\begin{array}{ll}
a_{i}w& \mbox{ if }i\in s/T+\Z,\\
0& \mbox{ if }i\not\in s/T+\Z.
\end{array}
\right.
\end{align}
By a similar argument, we have
$\glb{{\bf 1},a,w}{x_1,x_2}=Y_{M}(a,x_2)w$,
$\ty{s}{M}{{\bf 1},a}{x_2,x_1-x_2}(w)=\delta_{s,0}Y_{M}(a,x_2)w$ and hence
\begin{align}
\label{eq:ysonea}
\cty{s}{M}{{\bf 1},a}{k,l}(w)
&=
\delta_{s,0}\delta_{l,-1}a_{k}w.
\end{align}
 
\begin{lemma}\label{lemma:derivation}
Let $M$ be a $(V,T)$-module. Then, $Y_{M}(a_{-2}{\bf 1},x)=\frac{d}{dx}Y_{M}(a,x)$.
\end{lemma}
\begin{proof}
Let $a\in V, w\in M, k\in(1/T)\Z$ and $s\in\Z$ with $0\leq s\leq T-1$.
Taking $j=s/T$ and $l=-2$ in \eqref{eq:Borcherds-a-1}, we have
\begin{align}
\label{eq:dadx}
&\cty{s}{M}{a,{\bf 1}}{\frac{s}{T}+k,-2}w+\frac{s}{T} \cty{s}{M}{a,{\bf 1}}{\frac{s}{T}+k-1,-1}w\nonumber\\
&=
\left\{
\begin{array}{ll}
-ka_{s/T+k-1}w&\mbox{ if }k\in\Z,\\
0&\mbox{ if }k\not\in\Z.
\end{array}\right.
\end{align}
Let $r\in\Z$ with $0\leq r\leq T-1$ and $n\in r/T+\Z$.
By \eqref{eq:Borcherds-a-1-l} and \eqref{eq:dadx}, we have
\begin{align*}
(a_{-2}{\bf 1})_{n}w
&=\sum_{s=0}^{T-1}\cty{s}{M}{a,{\bf 1}}{n,-2}(w)\\
&=\sum_{s=0}^{T-1}\cty{s}{M}{a,{\bf 1}}{\frac{s}{T}+(-\frac{s}{T}+n),-2}(w)\\
&=\sum_{s=0}^{T-1}\frac{-s}{T} \cty{s}{M}{a,{\bf 1}}{\frac{s}{T}+(-\frac{s}{T}+n-1),-1)(w)-(-\frac{r}{T}+n}a_{n-1}w\\
&=\frac{-r}{T} a_{n-1}w-(-\frac{r}{T}+n)a_{n-1}w\\
&=-na_{n-1}w.
\end{align*}
\end{proof}

\section{\label{section:subspace}
Subspaces of $\C[\wz,\wz^{-1}]$
}
Throughout this section we fix a non-positive integer $\lw$.
This is the lowest weight of a graded vertex algebra $V=\oplus_{i=\lw}^{\infty}V_i$
which will be discussed in Section \ref{section:zhu}.
In this section we introduce
a subspace $\zhuOi{T}{\nthr}{\none}(\alpha,\beta ; \wz)$ of 
$\C[\wz,\wz^{-1}]$ (see \eqref{eq:def-otnm} below) for $n,m\in(1/T)\N$ and $\alpha,\beta\in\Z$
and we study its properties. The subspace $\zhuOi{T}{\nthr}{\none}(\alpha,\beta ; \wz)$
will be used to define the subspace
$\zhuOi{T}{\nthr}{\none}(V)$ of $V$ in Section \ref{section:zhu}.

For $N,q\in\Z$ and $Q\in\Q$, $O(N,Q,q;\wz)$ denotes the subspace of $\C[\wz,\wz^{-1}]$
spanned by
\begin{align}
\label{eq:res-x-q}
&\Res_{\wx}\big((1+\wx)^{Q}\wx^{q+j}\sum_{i\in\Z_{\leq N}}\wz^i\wx^{-i-1}\big)\nonumber\\
&=
\sum\limits_{i=0}^{N-q-j}\binom{Q}{i}\wz^{i+q+j},\quad j=0,-1,\ldots
\end{align}
and $\wz^i, i\in\Z_{\geq N+1}$.
We note that if $N\leq q$, then $O(N,Q,q;\wz)=\C[z,z^{-1}]$.
We also note that 
\begin{align}
\label{eq:osubset}
O(N,Q,q;\wz)\subset O(N,Q,q+1;\wz)\subset\cdots.
\end{align}
A similar computation as in the proof of Lemma \ref{lemma:express-ysab}
shows the following lemma (or see \cite[Proof of Lemma 2]{MT}).

\begin{lemma}\label{lemma:modulo}
Fix $N, q\in\Z$, $Q\in\Q$ and $i\in\Z_{\leq N}$. 
Then
\begin{align}
\label{eq:zimodbinom}
\wz^{i}&\equiv
\sum\limits_{k=1}^{N-q}\sum\limits_{j=1}^{k}\binom{-Q}{-i+q+j}\binom{Q}{k-j}
\wz^{q+k}\pmod{O(N,Q,q;\wz)}.
\end{align}
\end{lemma}

The proof of the following lemma is similar to that of \cite[Lemma 3]{MT}. 

\begin{lemma}\label{lemma:iso-oplus}
Let $N\in\Z, q_{0},\ldots,q_{T-1}\in\Z$ and $Q_{0},\ldots,Q_{T-1}\in \Q$ such that
$Q_i\not\equiv Q_j\pmod{\Z}$ for all $i\neq j$.
The diagonal map $\C[\wz,\wz^{-1}]\ni f\mapsto (f,\ldots,f)\in \C[\wz,\wz^{-1}]^{\oplus T}$
induces an isomorphism 
\begin{align*}
\C[\wz,\wz^{-1}]/\bigcap_{{\ws}=0}^{T-1}O(N,Q_{\ws},q_{\ws};\wz)
&\rightarrow \bigoplus_{{\ws}=0}^{T-1}\C[\wz,\wz^{-1}]/O(N,Q_{\ws},q_{\ws};\wz)
\end{align*}
as vector spaces.
\end{lemma}
\begin{proof}
It is sufficient to show that the induced map is surjective.
Note that $\C[\wz,\wz^{-1}]_{\geq N+1}$ is a subspace of $O(N,Q_{\ws},q_{\ws};\wz)$
for each $s$.
Fix an integer $q$ such that $q\leq\min\{q_0,\ldots,q_{T-1}\}$.
We may assume $q\leq N$ from the comment right after \eqref{eq:res-x-q}.
Since $O(N,Q_{\ws},q;\wz)$ is a subspace of $O(N,Q_{\ws},q_{{\ws}};\wz)$ for each ${\ws}=0,\ldots,T-1$,
it is sufficient to show that
the diagonal map 
\begin{align*}
&\C[\wz,\wz^{-1}]_{N+1-T(N-q),N}\ni f\\
&\mapsto (f+O(N,Q_{\ws},q;\wz))_{s=0}^{T-1}\in
\bigoplus_{{\ws}=0}^{T-1}\C[\wz,\wz^{-1}]/O(N,Q_{\ws},q;\wz)
\end{align*}
is surjective.
For a Laurent polynomial 
$\Lambda(\wz)=\sum_{i=N+1-T(N-q)}^{N}\lambda_{i}\wz^{i}\in 
\C[\wz,\wz^{-1}]_{N+1-T(N-q),N}$,
it follows by \eqref{eq:zimodbinom}
that 
\begin{align}
\label{eq:mod-poly-s}
\Lambda(\wz)
&\equiv
\sum\limits_{i=N+1-T(N-q)}^{N}\lambda_{i}
\sum\limits_{k=1}^{N-q}
\sum\limits_{j=1}^{k}\binom{-Q_{\ws}}{-i+q+j}\binom{Q_{\ws}}{k-j}\wz^{q+k}\nonumber\\
&\pmod{O(N,Q_{\ws},q;\wz)}
\end{align}
for $\ws=0,\ldots, T-1$.
We denote
$\sum_{j=1}^{k}\binom{-Q_{\ws}}{-i+q+j}\binom{Q_{\ws}}{k-j}$
by $\alpha^{{\ws},k}_{-i+q}$ for $0\leq \ws \leq T-1,1\leq k\leq N-q$ and $i\in\Z$.
Define $T$ $T(N-q)\times (N-q)$-matrices $\Gamma_{s},\ws=0,\ldots,T-1$ by
\begin{align*}
\Gamma_{{\ws}}&=
\begin{pmatrix}
\alpha^{{\ws},1}_{(T-1)(N-q)-1}&\alpha^{{\ws},2}_{(T-1)(N-q)-1}&\cdots&\alpha^{{\ws},N-q}_{(T-1)(N-q)-1}\\
\alpha^{{\ws},1}_{(T-1)(N-q)-2}&\alpha^{{\ws},2}_{(T-1)(N-q)-2}&\cdots&\alpha^{{\ws},N-q}_{(T-1)(N-q)-2}\\
\vdots&\vdots&&\vdots\\
\alpha^{{\ws},1}_{-N+q}&\alpha^{{\ws},2}_{-N+q}&\cdots&\alpha^{{\ws},N-q}_{-N+q}
\end{pmatrix}.
\end{align*}
Since
\begin{align*}
(\lambda_{N+1-T(N-q)},\lambda_{N+2-T(N-q)},\ldots,\lambda_{N})\Gamma_{{\ws}}\begin{pmatrix}\wz^{q+1}\\\vdots\\\wz^{N}
\end{pmatrix}
\end{align*}
is equal to the right-hand side of \eqref{eq:mod-poly-s} 
for ${\ws}=0,\ldots,T-1$,
it is sufficient to show that the square matrix 
\begin{align}
\label{eq:matrix-H}
\Gamma&=(\Gamma_0\ \Gamma_1\ \cdots\ \Gamma_{T-1})
\end{align}
of order $T(N-q)$ is non-singular.
It is proved in Subsection \ref{subsection:determinant} that $\Gamma$ is non-singular.
\end{proof}

For $N\in\Z$ and $\gamma\in\Q$, define
a linear automorphism $\varphi_{N,\gamma}$ of $\C[\wz,\wz^{-1}]$
by
\begin{align}
\label{eq:def-phi}
&\varphi_{N,\gamma}(\wz^i)\nonumber\\
&=
\left\{
\begin{array}{ll}
(-1)^{i+1}
\Res_{\wx}\big((1+\wx)^{\gamma-i}\wx^i\sum_{j\in\Z_{\leq N}}\wz^j\wx^{-j-1}\big)&
\mbox{for $i\leq N$,}\\
\wz^i & \mbox{for $i\geq N+1$.}
\end{array}\right.
\end{align}

\begin{lemma}\label{lemma:def-phi-nq}
\begin{align}
\label{eq:inv-x}
&\varphi_{N,\gamma}(\Res_{\wx}\big((1+\wx)^{k}\wx^{i}\sum_{j\in \Z_{\leq N}}\wz^j\wx^{-j-1}\big))\nonumber\\
&=
(-1)^{i+1}\Res_{\wx}\big((1+\wx)^{\gamma-k-i}\wx^{i}\sum_{j\in \Z_{\leq N}}\wz^j\wx^{-j-1}\big)
\end{align}
for $k\in \Q$ and $i\in\Z_{\leq N}$.
In particular, $\varphi_{N,\gamma}^2=\id_{\C[\wz,\wz^{-1}]}$
and
\begin{align*}
\varphi_{N,\gamma}(O(N,Q,q;\wz))=O(N,\gamma-Q-q,q;\wz)
\end{align*}
for $Q\in \Q$ and $q\in\Z$.
\end{lemma}
\begin{proof}
We simply write $\varphi=\varphi_{N,\gamma}$.
Let $i\in \Z_{\leq N}$.
Since
\begin{align*}
\varphi(\wz^i)&=
(-1)^{i+1}\sum_{j=0}^{N-i}\binom{\gamma-i}{j}z^{i+j},
\end{align*}
we have
\begin{align*}
&\varphi(\Res_{\wx}(1+\wx)^{k}\wx^{i}\sum_{j\in \Z_{\leq N}}\wz^j\wx^{-j-1})=
\sum_{j=0}^{N-i}\binom{k}{j}\varphi(z^{i+j})\\
&=
\sum_{j=0}^{N-i}\binom{k}{j}(-1)^{i+j+1}
\sum_{m=0}^{N-i-j}\binom{\gamma-i-j}{m}z^{i+j+m}\\
&=
\sum_{j=0}^{N-i}\binom{k}{j}(-1)^{i+j+1+m}
\sum_{m=0}^{N-i-j}\binom{-\gamma+i+j+m-1}{m}z^{i+j+m}\\
&=
\sum_{l=0}^{N-i}(-1)^{i+1+l}z^{i+l}
\sum_{\begin{subarray}{l}
0\leq j,m\leq N-i\\
j+m=l
\end{subarray}
}
\binom{k}{j}\binom{-\gamma+i+l-1}{m}\\
&=
\sum_{l=0}^{N-i}(-1)^{i+1+l}z^{i+l}
\binom{k-\gamma+i+l-1}{l}\\
&=(-1)^{i+1}
\sum_{l=0}^{N-i}\binom{-k+\gamma-i}{l}
z^{i+l}\\
&=
(-1)^{i+1}\Res_{\wx}\big((1+\wx)^{\gamma-k-i}\wx^{i}\sum_{j\in \Z_{\leq N}}\wz^j\wx^{-j-1}\big).
\end{align*}
By this, $\varphi^2(z^j)=z^j$ for $j\in\Z$. 
Since
\begin{align*}
&\varphi(\Res_{\wx}\big((1+\wx)^{Q}\wx^{q+d}\sum_{j\in\Z_{\leq N}}\wz^j\wx^{-j-1}\big)\\
&=(-1)^{q+d+1}\Res_{\wx}\big((1+\wx)^{\gamma-Q-q-d}\wx^{q+d}\sum_{j\in \Z_{\leq N}}\wz^j\wx^{-j-1}\big)\\
&=(-1)^{q+d+1}\sum_{m=0}^{-d}\binom{-d}{m}
\Res_{\wx}\big((1+\wx)^{\gamma-Q-q}\wx^{q+d+m}\sum_{j\in \Z_{\leq N}}\wz^j\wx^{-j-1}\big)
\end{align*}
for $d\in \Z_{\leq 0}$, we have 
$\varphi_{N,\gamma}(O(N,Q,q;\wz))=O(N,\gamma-Q-q,q;\wz)$.
\end{proof}

Throughout the rest of this section,
$m=l_1+i_1/T,p=l_2+i_2/T,n=l_3+i_3/T\in(1/T)\N$ with
$l_1,l_2,l_3\in\N$ and $0\leq i_1,i_2,i_3\leq T-1$.
We always denote $m,p,n$ as above until further notice.
For $i,j\in(1/T)\Z$,
$\rbk{i}{j}$ denotes the integer uniquely determined by 
the conditions 
\begin{align}
\label{eq:wrij}
0\leq \rbk{i}{j}\leq T-1\mbox{ and }i-j\equiv \frac{\rbk{i}{j}}{T}\pmod{\Z}.
\end{align} 
For $0\leq {\ws}\leq T-1$, $\ws^{\vee}$ denotes the integer uniquely determined by
the conditions 
\begin{align}
\label{eq:def-vee}
0\leq \ws^{\vee}\leq T-1\mbox{ and }
i_1-i_3\equiv \ws+\ws^{\vee}\pmod T.
\end{align}

For $\ws=0,\ldots,T-1$ and $\alpha,\beta\in\Z$, define 
\begin{align}
\label{eq:def-otnms}
\zhuOit{T}{\ws}{n}{m}(\alpha,\beta ; \wz)&=O(\alpha+\beta-1-\lw,\alpha-1+l_1+\delta(\ws\leq i_1)+\dfrac{\ws}{T},\nonumber\\
&\qquad -l_1-l_3-\delta(\ws\leq i_1)-\delta(T\leq \ws+i_3)-1;\wz)
\end{align}
and 
\begin{align}
\label{eq:def-otnm}
\zhuOi{T}{n}{m}(\alpha,\beta ; \wz)&=\bigcap_{s=0}^{T-1}
\zhuOit{T}{\ws}{n}{m}(\alpha,\beta ; \wz),
\end{align}
where $\lw$ is the fixed non-positive integer as stated at the beginning of this section
and $\delta(i\leq j)$ is defined in \eqref{eq:delta-function}.
For $\alpha,\beta\in\Z$, $j\in \Z_{\leq 0}$ and $s=0,\ldots,T-1$,
$\fo{T}{\ws}{\nthr}{\none}(\alpha,\beta,j;\wz)$
denotes the Laurent polynomial in $\zhuOit{T}{\ws}{n}{m}(\alpha,\beta ; \wz)$ defined by
\eqref{eq:res-x-q}, that is,
\begin{align}
\label{eq:def-o-poly}
&\fo{T}{\ws}{\nthr}{\none}(\alpha,\beta,j;\wz)\nonumber\\
&=
\Res_{\wx}\big((1+\wx)^{\alpha-1+l_1+\delta(s\leq i_1)+s/T}
\wx^{-l_1-l_3-\delta(s\leq i_1)-\delta(T\leq s+i_3)-1+j}\nonumber\\
&\quad{}\times
\sum_{\begin{subarray}{c}i\in\Z\\i\leq\alpha+\beta-1-\lw\end{subarray}}\wz^{i}\wx^{-i-1}\big)\nonumber\\
&=\sum_{i=0}^{\alpha+\beta-\lw+l_1+l_3+\delta(\ws\leq i_1)+\delta(T\leq \ws+i_3)-j}
\binom{\alpha-1+l_1+\delta(\ws\leq i_1)+\dfrac{\ws}{T}}{i}\nonumber\\
&\quad{}\times
\wz^{i-l_1-l_3-\delta(\ws\leq i_1)-\delta(T\leq \ws+i_3)-1+j}.
\end{align}
The disjoint union  $\{\fo{T}{\ws}{\nthr}{\none}(\alpha,\beta,j;\wz)\ |\ j=0,-1,\ldots\}
\cup \{\wz^i\ |\ i\geq \alpha+\beta-\lw\}$
spans $\zhuOit{T}{\ws}{n}{m}(\alpha,\beta;\wz)$.

\begin{lemma}
\label{lemma:descend}
Let 
$
m^{\prime}=l_1^{\prime}+i_1^{\prime}/T,
n^{\prime}=l_3^{\prime}+i_3^{\prime}/T\in(1/T)\N$ with $l_1^{\prime},l_3^{\prime}\in\N$
and $0\leq i_1^{\prime},i_3^{\prime}\leq T-1$.
If $m^{\prime}\leq m$ and $n^{\prime}\leq n$,
then $\zhuOit{T}{\ws}{n}{m}(\alpha,\beta ; \wz)
\subset
\zhuOit{T}{\ws}{\nthr^{\prime}}{\none^{\prime}}(\alpha,\beta ; \wz)$ for $\alpha,\beta\in\Z$ and $s=0,\ldots,T-1$.
In particular, $\zhuOi{T}{n}{m}(\alpha,\beta ; \wz)
\subset \zhuOi{T}{\nthr^{\prime}}{\none^{\prime}}(\alpha,\beta ; \wz)$.
\end{lemma}
\begin{proof}
Let $
\rho_1=l_1+\delta(\ws\leq i_1)-(l_1^{\prime}+\delta(\ws\leq i_1^{\prime}))$
and $\rho_3=l_3+\delta(T\leq \ws+i_3)-(l_3^{\prime}+\delta(T\leq \ws+i_3^{\prime}))$
for $s=0,\ldots,T-1$.
It follows by $m^{\prime}\leq m$ and $n^{\prime}\leq n$ that 
$\rho_1$ and $\rho_3$ are non-negative integers.
Since
\begin{align*}
\fo{T}{\ws}{\nthr}{\none}(\alpha,\beta,j;\wz)
&=\sum_{i=0}^{\rho_1}\binom{\rho_1}{i}\fo{T}{\ws}{\nthr^{\prime}}{\none^{\prime}}(\alpha,\beta,j-\rho_1-\rho_3+i;\wz),
\end{align*}
the proof is complete.
\end{proof}

A direct computation shows 
\begin{align}
\label{eq:ssvee}
\dfrac{-s-s^{\vee}+i_1-i_3}{T}+\delta(T\leq s^{\vee}+i_3)&=\delta(s\leq i_1)-1
\end{align}
for $s=0,\ldots,T-1$ and hence
\begin{align}
\label{eq:svee^2-s^2}
\delta(s^{\vee}\leq i_1)+\delta(T\leq s^{\vee}+i_3)=\delta(s\leq i_1)+\delta(T\leq s+i_3).
\end{align}
For a non-positive integer $j$, it follows by \eqref{eq:inv-x},\eqref{eq:ssvee}
and \eqref{eq:svee^2-s^2}
that
\begin{align*}
&\varphi_{\alpha+\beta-1-\lw,\alpha+\beta+m-n-2}(\fo{T}{\ws^{\vee}}{\nthr}{\none}(\beta,\alpha,j;\wz))\nonumber\\
&=(-1)^{-l_1-l_3-\delta(s\leq i_1)-\delta(T\leq s+i_3)+j}\\
&\quad{}\times
\Res_{\wx}\big((1+\wx)^{\alpha-1+l_1+\delta(\ws\leq i_1)+\ws/T-j}
\wx^{-l_1-l_3-\delta(s\leq i_1)-\delta(T\leq s+i_3)-1+j}\nonumber\\
&\quad{}\times
\sum_{\begin{subarray}{c}i\in\Z\\i\leq\alpha+\beta-1-\lw\end{subarray}}\wz^{i}\wx^{-i-1}\big)\nonumber\\
&=(-1)^{-l_1-l_3-\delta(s\leq i_1)-\delta(T\leq s+i_3)+j}\sum_{k=0}^{-j}\binom{-j}{k}
\fo{T}{\ws}{\nthr}{\none}(\alpha,\beta,j+k;\wz)
\end{align*}
and hence 
\begin{align}
\label{eq:phi-ab}
\varphi_{\alpha+\beta-1-\lw,\alpha+\beta+m-n-2}(\zhuOit{T}{\ws^{\vee}}{\nthr}{\none}(\beta,\alpha ; \wz))
&=\zhuOit{T}{\ws}{n}{m}(\alpha,\beta;\wz).
\end{align}
Thus,
$\varphi_{\alpha+\beta-1-\lw,\alpha+\beta+m-n-2}$ induces an isomorphism
\begin{align}
\label{eq:poly-o-phi-s}
\C[\wz,\wz^{-1}]/\zhuOit{T}{\ws^{\vee}}{\nthr}{\none}(\beta,\alpha ; \wz)
&\rightarrow
\C[\wz,\wz^{-1}]/\zhuOit{T}{\ws}{n}{m}(\alpha,\beta ; \wz)
\end{align}
and hence
\begin{align}
\label{eq:poly-o-phi}
\C[\wz,\wz^{-1}]/\bigcap_{\ws=0}^{T-1}\zhuOit{T}{\ws}{n}{m}(\beta,\alpha ; \wz)&\cong
\C[\wz,\wz^{-1}]/\bigcap_{\ws=0}^{T-1}\zhuOit{T}{\ws}{n}{m}(\alpha,\beta ; \wz)
\end{align}
by Lemma \ref{lemma:modulo}.

For $\wr=0,\dots,T-1$ and $i\in\Z_{\leq \alpha+\beta-1-\lw}$, it follows by the argument in the proof of 
Lemma \ref{lemma:iso-oplus}
that there exists a unique Laurent polynomial in 
$\C[z,z^{-1}]_{\alpha+\beta-\lw-T(\alpha+\beta-\lw+l_1+l_3+2),\alpha+\beta-1-\lw}$,
which we denote by 
$\unitmu{T}{\wr}{\nthr}{\none}(\alpha,\beta,i;\wz)$, such that
\begin{align}
\label{eq:unit-s}
&\unitmu{T}{\wr}{\nthr}{\none}(\alpha,\beta,i;\wz)
\equiv \delta_{r,s}z^i\nonumber\\
&\pmod{
O(\alpha+\beta-1-\lw,\alpha-1+l_1+\delta(\ws\leq i_1)+\frac{\ws}{T},-l_1-l_3-3; \wz)},\nonumber\\
&\qquad s=0,\ldots,T-1.
\end{align}
We also define 
\begin{align}
\label{eq:unit-zero}
\unitmu{T}{\wr}{\nthr}{\none}(\alpha,\beta,i;\wz)&=0\mbox{ for $i\in\Z_{\geq \alpha+\beta-\lw}$}
\end{align}
for convenience.
Since 
\begin{align*}
-l_1-l_3-3\leq -l_1-l_3-\delta(\ws\leq i_1)-\delta(T\leq \ws+i_3)-1,
\end{align*}
it follows by \eqref{eq:osubset}, \eqref{eq:unit-s} and \eqref{eq:unit-zero} that
\begin{align}
\label{eq:unit-qs}
\unitmu{T}{\wr}{\nthr}{\none}(\alpha,\beta,i;\wz)&\equiv 
\delta_{r,s}\wz^i\pmod{\zhuOit{T}{\ws}{\nthr}{\none}(\alpha,\beta;\wz)}
\end{align}
for $r,s=0,\ldots,T-1$ and $i\in\Z$.
It follows from Lemma \ref{lemma:iso-oplus}
that 
\begin{align}
\label{eq:sum-e-s}
\sum_{s=0}^{T-1}
\unitmu{T}{\ws}{\nthr}{\none}(\alpha,\beta,i;\wz)\equiv \wz^i
\pmod{\zhuOi{T}{n}{m}(\alpha,\beta;\wz)}
\end{align}
for $i\in\Z$.
Define a Laurent polynomial $\pmul{T}{\nthr}{\ntwo}{\none}(\alpha,\beta;\wz)$ by
\begin{align}
\label{eq:def-mul-x}
&\pmul{T}{\nthr}{\ntwo}{\none}(\alpha,\beta;\wz)\nonumber\\
&=\sum_{i=0}^{l_2}
\binom{-l_1-l_3+l_2-\delta(\rbk{p}{n}\leq i_1)-\delta(T\leq \rbk{p}{n}+i_3)}{i}\nonumber\\
&\quad{}\times
\Res_{\wx}\big((1+\wx)^{\alpha-1+l_1+\delta({\rbk{p}{n}}\leq i_1)+\rbk{p}{n}/T}\nonumber\\
&\quad{}\times
\wx^{-l_1-l_3+l_2-\delta(\rbk{p}{n}\leq i_1)-\delta(T\leq \rbk{p}{n}+i_3)-i}
\sum_{j\in\Z}
\unitmu{T}{\rbk{p}{n}}{\nthr}{\none}(\alpha,\beta,j;\wz)\wx^{-j-1}
\big)\nonumber\\
&\in\C[z,z^{-1}]_{\alpha+\beta-\lw-T(\alpha+\beta-\lw+l_1+l_3+2),\alpha+\beta-1-\lw},
\end{align}
where $\rbk{p}{n}$ is defined in \eqref{eq:wrij}.
This is used to define the product $*^{T}_{n,p,m}$ on a vertex algebra  
in Section \ref{section:zhu}.

We denote $\Span_{\C}\{ \wz^i\in \C[\wz,\wz^{-1}]\ |\ i\neq -1\}$ by $\C[\wz,\wz^{-1}]_{\neq -1}$.
The following two results will be used to compute ${\bf 1}*^{T}_{n,p,m}a$ for $a\in V$ in
Section \ref{section:zhu}.
 
\begin{lemma}\label{lemma:n-neq-p-zero}
Let $\alpha,r\in\Z$ with $0\leq r\leq T-1$. 
Then 
\begin{align*}
\sum_{j\in\Z}
\unitmu{T}{\wr}{\nthr}{\none}(0,\alpha,j;\wz)\wx^{-j-1}
&\equiv
\delta_{r,0}\wz^{-1}
\pmod{(\C[\wz,\wz^{-1}]_{\neq -1})\db{\wx}}.
\end{align*}
\end{lemma}
\begin{proof}
Since 
\begin{align*}
\fo{T}{0}{\nthr}{\none}(0,\alpha,j;\wz)
&=\sum_{i=0}^{\alpha+\lw+l_1+l_3+1-j}
\binom{l_1}{i}\wz^{i-l_1-l_3-2+j}\in\C[\wz,\wz^{-1}]_{\leq -2}
\end{align*}
for all $j\in\Z_{\leq 0}$,
$\zhuOit{T}{0}{\nthr}{\none}(0,\alpha;\wz)$ is a subspace of 
$\C[\wz,\wz^{-1}]_{\neq -1}$.
By \eqref{eq:unit-qs},
we have the desired result.
\end{proof}

\begin{lemma}\label{lemma:1-poly}
For $\alpha\in\Z$, we have
\begin{align*}
\pmul{T}{\nthr}{\ntwo}{\none}(0,\alpha;\wz)
&\equiv \delta_{n,p}\wz^{-1}
\pmod{\C[\wz,\wz^{-1}]_{\neq -1}}.
\end{align*}
\end{lemma}

\begin{proof}
If $n\not\equiv p\pmod{\Z}$, then it follows by Lemma \ref{lemma:n-neq-p-zero} that
\begin{align*}
\pmul{T}{\nthr}{\ntwo}{\none}(0,\alpha;\wz)
\equiv 0\pmod{\C[\wz,\wz^{-1}]_{\neq -1}}.
\end{align*}
Suppose $n\equiv p\pmod{\Z}$. By Lemma \ref{lemma:n-neq-p-zero} again,
the same computation as in the proof of \cite[Lemma 4.7]{DJ1} shows 
\begin{align*}
\pmul{T}{\nthr}{\ntwo}{\none}(0,\alpha;\wz)
&\equiv \delta_{n,p}\wz^{-1}
\pmod{\C[\wz,\wz^{-1}]_{\neq -1}}.
\end{align*}
\end{proof}

The following result will be used in order to obtain Lemma \ref{lemma:ab-ba-V},
which induces the commutator formula in Lemma \ref{lemma:commutative}. 
\begin{lemma}\label{lemma:comm-x}
For $\alpha,\beta\in\Z$, we have
\begin{align}
\label{eq:comm-poly}
&\pmul{T}{\nthr}{\ntwo}{\none}(\alpha,\beta;\wz)-\varphi_{\alpha+\beta-1-\lw,\alpha+\beta+m-n-2}(
\pmul{T}{\nthr}{\none+\nthr-\ntwo}{\none}(\beta,\alpha;\wz))\nonumber\\
&\quad{}-\Res_{\wx}(1+\wx)^{\alpha-1+p-n}
\sum_{j\in\Z}\unitmu{T}{\rbk{p}{n}}{\nthr}{\none}(\alpha,\beta,j;\wz)\wx^{-j-1}
\in \zhuOi{T}{n}{m}(\alpha,\beta;\wz).
\end{align}
\end{lemma}
\begin{proof}
The proof is similar to that of \cite[Lemma 3.4]{DJ2}.
We simply write $\wr=\rbk{p}{n}$ and $\varphi=\varphi_{\alpha+\beta-1-\lw,\alpha+\beta+m-n-2}$.
It follows by
\begin{align*}
(m+n-p)-n\equiv \frac{i_1-i_2}{T}\equiv \frac{\wr^{\vee}}{T} \pmod{\Z}
\end{align*}
that
$\pmul{T}{\nthr}{\none+\nthr-\ntwo}{\ntwo}(\beta,\alpha;\wz)
\in \cap_{\ws\neq \wr^{\vee}}\zhuOit{T}{\ws}{n}{m}(\beta,\alpha ; \wz)$,
where $\wr^{\vee}$ is defined in \eqref{eq:def-vee}.
Since 
$\varphi(\pmul{T}{\nthr}{\none+\nthr-\ntwo}{\ntwo}(\beta,\alpha;\wz))\in \cap_{\ws\neq \wr}\zhuOit{T}{\ws}{n}{m}(\alpha,\beta ; \wz)$
by \eqref{eq:phi-ab}, we have
\begin{align*}
&\pmul{T}{\nthr}{\ntwo}{\none}(\alpha,\beta;\wz)-\varphi(
\pmul{T}{\nthr}{\none+\nthr-\ntwo}{\none}(\beta,\alpha;\wz))\\
&\quad{}-\Res_{\wx}(1+\wx)^{\alpha-1+p-n}
\sum_{j\in\Z}\unitmu{T}{\wr}{\nthr}{\none}(\alpha,\beta,j;\wz)\wx^{-j-1}
\in \bigcap_{\ws\neq \wr}\zhuOit{T}{\ws}{n}{m}(\alpha,\beta ; \wz).
\end{align*}
Thus,
it is sufficient to show 
\eqref{eq:comm-poly} modulo $\zhuOit{T}{\wr}{\nthr}{\none}(\alpha,\beta ; \wz)$
by Lemma \ref{lemma:iso-oplus}.
Define
\begin{align}
\label{eq:ab-ba-epsilon}
\varepsilon&=\left\{\begin{array}{ll}
1&\mbox{if }T\leq i_1+i_3-i_2,\\
0&\mbox{if }0\leq i_1+i_3-i_2<T,\\
-1&\mbox{if }i_1+i_3-i_2<0.
\end{array}\right.
\end{align}
It follows by the formula of $\varepsilon$ in the proof of \cite[Lemma 3.4]{DJ2}
and \eqref{eq:svee^2-s^2} that
\begin{align*}
&\pmul{T}{\nthr}{\none+\nthr-\ntwo}{\none}(\beta,\alpha;\wz)\\
&=\sum_{i=0}^{l_1+l_3-l_2+\varepsilon}\binom{-l_1-l_3+(l_1+l_3-l_2+\varepsilon)-\delta(\wr^{\vee}\leq i_1)-\delta(T\leq \wr^{\vee}+i_3)}{i}\\
&\qquad{}\times\Res_{\wx}
(1+\wx)^{\beta-1+l_1+\delta(\wr^{\vee}\leq i_1)+\wr^{\vee}/T}\\
&\qquad{}\times \wx^{-l_1-l_3+(l_1+l_3-l_2+\varepsilon)
-\delta(\wr^{\vee}\leq i_1)-\delta(T\leq \wr^{\vee}+i_3)-i}
\sum_{j\in\Z}\unitmu{T}{\wr^{\vee}}{\nthr}{\none}(\beta,\alpha,j;\wz)\wx^{-j-1}\\
&=\sum_{i=0}^{l_1+l_3-l_2+\varepsilon}\binom{-l_2-1}{i}\Res_{\wx}
(1+\wx)^{\beta-1+l_1+\delta(\wr^{\vee}\leq i_1)+\wr^{\vee}/T}\wx^{-l_2-1-i}\\
&\qquad{}\times
\sum_{j\in\Z}\unitmu{T}{\wr^{\vee}}{\nthr}{\none}(\beta,\alpha,j;\wz)\wx^{-j-1}.
\end{align*}
Thus, it follows by \eqref{eq:inv-x} that
\begin{align}
\label{eq:ab-ba}
&
\varphi(\pmul{T}{\nthr}{\none+\nthr-\ntwo}{\none}(\beta,\alpha;\wz))\nonumber\\
&=\sum_{i=0}^{l_1+l_3-l_2+\varepsilon}\binom{-l_2-1}{i}(-1)^{-l_2-i}\Res_{\wx}
(1+\wx)^{\alpha-1+p-n+i}\wx^{-l_2-1-i}\nonumber\\
&\qquad{}\times
\sum_{j\in\Z}\unitmu{T}{\wr}{\nthr}{\none}(\alpha,\beta,j;\wz)\wx^{-j-1}
\end{align}
and therefore
\begin{align*}
&
\varphi(\pmul{T}{\nthr}{\none+\nthr-\ntwo}{\none}(\beta,\alpha;\wz))\\
&\equiv\sum_{i=0}^{l_1+l_3-l_2+\varepsilon}\binom{-l_2-1}{i}
(-1)^{-l_2-i}
\Res_{\wx}
(1+\wx)^{\alpha-1+p-n+i}\wx^{-l_2-1-i}\\
&\quad{}\times\sum_{\begin{subarray}{c}j\in\Z\\j\leq \alpha+\beta-1-\lw
\end{subarray}}
\wz^{j}\wx^{-j-1}
\qquad\pmod{\zhuOit{T}{\wr}{\nthr}{\none}(\alpha,\beta ; \wz)}.
\end{align*}

The same argument as in the proof of \cite[Lemma 3.4]{DJ2} shows
\begin{align*}
&\pmul{T}{\nthr}{\ntwo}{\none}(\alpha,\beta;\wz)-\varphi(
\pmul{T}{\nthr}{\none+\nthr-\ntwo}{\none}(\beta,\alpha;\wz))\\
&\equiv
\sum_{i=0}^{l_2}
\binom{-l_1-l_3+l_2-\varepsilon-1}{i}\\
&\qquad{}\times
\Res_{\wx}\big(
(1+\wx)^{\alpha-1+l_1+\delta(r\leq i_1)+r/T}
\wx^{-l_1-l_3+l_2-\varepsilon-1-i}
\sum_{\begin{subarray}{c}j\in\Z\\j\leq \alpha+\beta-1-\lw\end{subarray}}
\wz^j\wx^{-j-1}\big)\\
&\quad{}-\sum_{i=0}^{l_1+l_3-l_2+\varepsilon}\binom{-l_2-1}{i}(-1)^{-l_2-i}\Res_{\wx}
(1+\wx)^{\alpha-1+p-n+i}\wx^{-l_2-1-i}\sum_{\begin{subarray}{c}j\in\Z\\j\leq \alpha+\beta-1-\lw\end{subarray}}
\wz^j\wx^{-j-1}\\
&\qquad\pmod{\zhuOit{T}{\wr}{\nthr}{\none}(\alpha,\beta ; \wz)}\\
&=\Res_{\wx}(1+\wx)^{\alpha-1+p-n}\sum_{\begin{subarray}{c}j\in\Z\\j\leq \alpha+\beta-1-\lw\end{subarray}}
\wz^j\wx^{-j-1}.
\end{align*}
The proof is complete.
\end{proof}

Let $l\in(1/T)\N$ with $l\leq n,m$.
Then, it follows by Lemma \ref{lemma:descend} that
\begin{align*}
\unitmu{T}{\wr}{n}{m}(\alpha,\beta,i;\wz)
\equiv
\unitmu{T}{\wr}{n-l}{m-l}(\alpha,\beta,i;\wz)
\pmod{\zhuOi{T}{n-l}{m-l}(\alpha,\beta;\wz)}
\end{align*}
for $\alpha,\beta,i\in\Z$.
The same computation as in the proof of \cite[Proposition 4.3]{DJ2} shows the following lemma.
\begin{lemma}
\label{lemma:multi-descend}
Let $l\in(1/T)\N$ with $l\leq n,m$.
Then
\begin{align*}
\pmul{T}{\nthr}{\ntwo}{\none}(\alpha,\beta;\wz)\equiv
\pmul{T}{\nthr-l}{\ntwo-l}{\none-l}(\alpha,\beta;\wz)\pmod{\zhuOi{T}{n-l}{m-l}(\alpha,\beta;\wz)}
\end{align*}
for $\alpha,\beta\in\Z$.
\end{lemma}

Let $T^{\prime}$ be a positive multiple of $T$ and $\alpha,\beta\in\Z$.
Set $d=T^{\prime}/T$.
We note that
$
m=l_1+di_1/T^{\prime},
p=l_2+di_2/T^{\prime}$ and 
$n=l_3+di_3/T^{\prime}$.
Thus it follows by \eqref{eq:def-otnms}
that
\begin{align}
\label{eq:Tprime-T}
\zhuOit{T^{\prime}}{dr}{n}{m}(\alpha,\beta;z)&=
\zhuOit{T}{r}{n}{m}(\alpha,\beta;z)
\end{align}
for  $r=0,\ldots,T-1$.
By this and \eqref{eq:unit-qs}, we have
\begin{align*}
\unitmu{T^{\prime}}{dr}{n}{m}(\alpha,\beta,i;\wz)&\equiv
\delta_{r,s}z^i
\pmod{\zhuOit{T}{s}{n}{m}(\alpha,\beta,i;\wz)}
\end{align*}
for $i\in\Z$ and $r,s=0,\ldots,T-1$. Therefore, Lemma \ref{lemma:iso-oplus}
implies
\begin{align*}
\unitmu{T^{\prime}}{dr}{n}{m}(\alpha,\beta,i;\wz)&\equiv
\unitmu{T}{r}{n}{m}(\alpha,\beta,i;\wz)\pmod{\zhuOi{T}{n}{m}(\alpha,\beta;z)}
\end{align*}
for $i\in\Z$ and $r=0,\ldots,T-1$. 
By \eqref{eq:def-mul-x}, 
we have the following result.
\begin{lemma}\label{lemma:mul-Tprime-T}
Let $T^{\prime}$ be a positive multiple of $T$ and $\alpha,\beta\in\Z$.
Then
\begin{align*}
\pmul{T^{\prime}}{\nthr}{\ntwo}{\none}(\alpha,\beta;\wz)\equiv
\pmul{T}{\nthr}{\ntwo}{\none}(\alpha,\beta;\wz)\pmod{\zhuOi{T}{n}{m}(\alpha,\beta;z)}.
\end{align*}
\end{lemma}

\section{\label{section:zhu}
Associative algebras $\zhualg{T}{m}(V)$ and bimodules $\zhumod{T}{n}{m}(V)$}

Throughout the rest of this paper, 
we always assume the following properties for a vertex algebra $V$:
$V$ has a grading $V=\oplus_{i=\lw}^{\infty}V_i$ such that $\lw\in\Z_{\leq 0}$,
${\bf 1}\in V_0$ and for any homogeneous element $a\in V$,
$a_{i}V_j\subset V_{\wt a-1-i+j}$,
where $V_i=0$ for $i<\Delta$. 
Every vertex operator algebra satisfies these properties.
Throughout this section, we fix $m=l_1+i_1/T,p=l_2+i_2/T,n=l_3+i_3/T\in(1/T)\N$ with
$l_1,l_2,l_3\in\N$ and $0\leq i_1,i_2,i_3\leq T-1$.

In this section, 
we first define a product $*^{T}_{n,p,m}$ on $V$
and a quotient space $\zhumod{T}{n}{m}(V)$ of $V$.
In the following,
we shall use a similar argument as in \cite[Section 3]{DJ1}.
For $a\in V_i$, we denote $i$ by $\wt a$.
Define 
\begin{align}
\label{eq:def-unit-ab}
\hunitmu{T}{\ws}{\nthr}{\none}(a,b,i)=\unitmu{T}{\ws}{\nthr}{\none}(\wt a,\wt b,i;\wz)|_{\wz^j=a_jb}\in V
\end{align}
for homogeneous elements $a,b$ of $V$ and $i\in\Z$, 
where $\unitmu{T}{\ws}{\nthr}{\none}(\wt a,\wt b,i;\wz)$ is defined in \eqref{eq:unit-s}, 
and extend $\hunitmu{T}{\ws}{\nthr}{\none}(a,b,i)$ for arbitrary $a,b\in V$
by linearity.

Let 
$\zhuOzero{T}{n}{m}(V)$
be the subspace of $V$ spanned by
\begin{align}
\label{eq:span-a-1}
\{a_{-2}{\bf 1}+(\wt a+m-n)a\in V\ |\ \mbox{homogeneous }a\in V\}
\end{align}
and $\zhuOi{T}{n}{m}(V)$ the subspace of $V$ spanned by
\begin{align}
\label{eq:span-a-b}
&\Big\{P(\wz)|_{\wz^j=a_jb}\in V\ \Big|\ 
\begin{array}{l}
\mbox{homogeneous $a,b\in V$ and}\\
P(\wz)\in \zhuOi{T}{n}{m}(\wt a,\wt b ; \wz)
\end{array}\Big\}.
\end{align}

A similar argument as in the proof of \cite[Lemma 2.1.3]{Z} shows the following lemma as stated in 
the proof of \cite[Lemma 2.3]{DJ1}.
\begin{lemma}\label{lemma:inv-ab}
For homogeneous $a,b\in V$, we have
\begin{align*}
&\Res_{\wx}(1+\wx)^{i}\wx^{j}Y(b,\wx)a\\
&\equiv
(-1)^{j+1}\Res_{\wx}(1+\wx)^{\wt a+\wt b+m-n-2-i-j}
\wx^{j}Y(a,\wx)b
\pmod{\zhuOzero{T}{n}{m}(V)}.
\end{align*}
for $i\in\Q,j\in\Z$ and homogeneous $a,b\in V$.
\end{lemma}

By \eqref{eq:sum-e-s}, we have
\begin{align}
\label{eq:sum-unit-s}
\sum_{\ws=0}^{T-1}\hunitmu{T}{\ws}{\nthr}{\none}(a,b,i)&\equiv a_{i}b
\pmod{\zhuOi{T}{n}{m}(V)}
\end{align}
for $i\in\Z$.
Define 
\begin{align}
\label{eq:multi-ab}
a*^{T}_{\nthr,\ntwo,\none}b&=\pmul{T}{\nthr}{\ntwo}{\none}(\wt a,\wt b;\wz)|_{\wz^j=a_jb}\in V
\end{align}
for homogeneous $a,b\in V$,
where $\pmul{T}{\nthr}{\ntwo}{\none}$ is defined in \eqref{eq:def-mul-x}, and
extend $a*^{T}_{\nthr,\ntwo,\none}b$ for arbitrary $a,b\in V$ by linearity.
By $Y({\bf 1},\wx)=\id_{V}$ and Lemma \ref{lemma:1-poly}, we have
\begin{align}
\label{eq:identity}
{\bf 1}*^{T}_{n,p,m}a&=\delta_{n,p}a
\end{align}
for $a\in V$.

\begin{definition}\label{definition:O(V)}
Let 
$\zhuOii{T}{\nthr}{\none}(V)$ be 
the subspace of $V$ spanned by
\begin{align*}
u*^{T}_{n,p_3,m}((a*^{T}_{p_3,p_2,p_1}b)*^{T}_{p_3,p_1,m}c-a*^{T}_{p_3,p_2,m}(b*^{T}_{p_2,p_1,m}c))
\end{align*}
for all $a,b,c,u\in V$ and all $p_1,p_2,p_3\in (1/T)\N$.
Define
\begin{align*}
&\zhuOiii{T}{n}{m}(V)\\
&=
\sum_{p_1,p_2\in(1/T)\N}(V*^{T}_{n,p_2,p_1}(\zhuOzero{T}{p_2}{p_1}(V)+
\zhuOi{T}{p_2}{p_1}(V))*^{T}_{n,p_1,m}V
\end{align*}
and 
\begin{align*}
\zhuO{T}{n}{m}(V)&=
\zhuOzero{T}{n}{m}(V)+
\zhuOi{T}{n}{m}(V)+
\zhuOii{T}{n}{m}(V)+
\zhuOiii{T}{n}{m}(V).
\end{align*}
\end{definition}

By \eqref{eq:identity}, we have
\begin{align*}
(a*^{T}_{n,p_2,p_1}b)*^{T}_{n,p_1,m}c-a*^{T}_{n,p_2,m}(b*^{T}_{p_2,p_1,m}c)\in 
\zhuOii{T}{\nthr}{\none}(V)
\end{align*}
for $a,b,c\in V$ and $p_1,p_2\in (1/T)\N$.

\begin{lemma}\label{lemma:ab-ba-V}
For $a,b\in V$, we have
\begin{align*}
&a*_{n,p,m}^{T}b-b*_{n,m+n-p,m}^{T}a\\
&\quad{}-\Res_{\wx}(1+\wx)^{\wt a-1+p-n}
\sum_{\begin{subarray}{c}j\in\Z
\end{subarray}}
\hunitmu{T}{\rbk{p}{n}}{\nthr}{\none}(a,b,j)\wx^{-j-1}
\\
&\in \zhuOzero{T}{n}{m}(V)+\zhuOi{T}{n}{m}(V),
\end{align*}
where $\rbk{p}{n}$ is defined in \eqref{eq:wrij}.
\end{lemma}
\begin{proof}
We may assume $a$ and $b$ to be homogeneous elements of $V$.
We simply write $\wr=\rbk{p}{n}$.
Let $\varepsilon$ be the integer defined in \eqref{eq:ab-ba-epsilon}.
By Lemma \ref{lemma:inv-ab} and \eqref{eq:ab-ba}, we have
\begin{align*}
& b*^{T}_{n,m+p-n,m}a\\
&\equiv\sum_{i=0}^{l_1+l_3-l_2+\varepsilon}\binom{-l_2-1}{i}(-1)^{-l_2-i}
\Res_{\wx}
(1+\wx)^{\wt a-1+p-n+i}\wx^{-l_2-1-i}\\
&\quad{}\times
\sum_{\begin{subarray}{c}j\in\Z
\end{subarray}}
\hunitmu{T}{\wr}{\nthr}{\none}(a,b,j)\wx^{-j-1}
\qquad\pmod{\zhuOzero{T}{n}{m}(V)+\zhuOi{T}{n}{m}(V)}\\
&=\sum_{i=0}^{l_1+l_3-l_2+\varepsilon}\binom{-l_2-1}{i}(-1)^{-l_2-i}\Res_{\wx}
(1+\wx)^{\wt a-1+p-n+i}\wx^{-l_2-1-i}\\
&\qquad\times
\sum_{\begin{subarray}{c}j\in\Z
\end{subarray}}
\unitmu{T}{\wr}{n}{m}(\wt a,\wt b,j;\wz)|_{\wz^k=a_{k}b}\ \wx^{-j-1}\\
&=
\varphi_{\wt a+\wt b-1-\lw,\wt a+\wt b+m-n-2}(\pmul{T}{\nthr}{\none+\nthr-\ntwo}{\none}(\wt b,\wt a;\wz))|_{\wz^k=a_{k}b},
\end{align*}
where $\varphi_{\wt a+\wt b-1-\lw,\wt a+\wt b+m-n-2}$ is defined by \eqref{eq:def-phi}.
Thus, the assertion follows from Lemma \ref{lemma:comm-x}.
\end{proof}

By \eqref{eq:identity} and 
Lemmas \ref{lemma:n-neq-p-zero} and \ref{lemma:ab-ba-V},
we have
\begin{align}
\label{eq:a-multi-1}
a*^{T}_{n,m,m}{\bf 1}&\equiv a\pmod{\zhuOzero{T}{n}{m}(V)+\zhuOi{T}{n}{m}(V)}
\end{align}
for $a\in V$.

The same argument as in the proof of \cite[Lemma 3.8]{DJ2} shows
the following lemma.
\begin{lemma}\label{lemma:ass-O}
For $m,p,n\in(1/T)\Z$, we have
$V*^{T}_{n,p,m}\zhuO{T}{p}{m}(V)\subset \zhuO{T}{n}{m}(V)$ and
$\zhuO{T}{n}{p}(V)*^{T}_{n,p,m}V\subset \zhuO{T}{n}{m}(V)$.
\end{lemma}

We define 
\begin{align}
\label{eq:def-atnmv}
\zhumod{T}{n}{m}(V)&=V/\zhuO{T}{n}{m}(V).
\end{align}
If $m=n$, we simply write $\zhualg{T}{m}(V)=\zhumod{T}{m}{m}(V)$.
By Definition \ref{definition:O(V)}, \eqref{eq:identity}, \eqref{eq:a-multi-1}
and Lemma \ref{lemma:ass-O},
we have the following result.

\begin{theorem}\label{theorem:zhu-algebra}
Let $m,n\in(1/T)\N$.
Then, $(\zhualg{T}{m}(V),*^{T}_{m,m,m})$ is an associative $\C$-algebra and
$\zhumod{T}{n}{m}(V)$ is an $\zhualg{T}{n}(V)$-$\zhualg{T}{m}(V)$-bimodule,
where the left action of $\zhualg{T}{n}(V)$ is given by $*^{T}_{n,n,m}$
and the right action of $\zhualg{T}{m}(V)$ is given by $*^{T}_{n,m,m}$.
\end{theorem}

Lemmas \ref{lemma:descend} and \ref{lemma:multi-descend} 
imply the following result.
\begin{proposition}
Let $l,m,n\in(1/T)\N$ with $l\leq n,m$.
Then $\zhuOi{T}{n}{m}(V)$ is a subspace of $\zhuOi{T}{n-l}{m-l}(V)$.
Moreover, the identity map on $V$ induces a surjective algebra homomorphism 
$\zhualg{T}{m}(V)\rightarrow \zhualg{T}{m-l}(V)$ and
a surjective $\zhualg{T}{n}(V)$-$\zhualg{T}{m}(V)$-bimodule homomorphism 
$\zhumod{T}{n}{m}(V)\rightarrow \zhumod{T}{n-l}{m-l}(V)$.
\end{proposition}

Lemma \ref{lemma:mul-Tprime-T} and \eqref{eq:Tprime-T}
imply the following result.
\begin{proposition}
Let $m,n\in (1/T)\N$ and $T^{\prime}$ a positive multiple of $T$.
Then $\zhuOi{T}{n}{m}(V)$ is a subspace of $\zhuOi{T^{\prime}}{n}{m}(V)$.
Moreover, the identity map on $V$ induces a surjective algebra homomorphism 
$\zhualg{T^{\prime}}{m}(V)\rightarrow \zhualg{T}{m}(V)$ and
a surjective $\zhualg{T^{\prime}}{n}(V)$-$\zhualg{T^{\prime}}{m}(V)$-bimodule homomorphism 
$\zhumod{T^{\prime}}{n}{m}(V)\rightarrow \zhumod{T}{n}{m}(V)$.
\end{proposition}

\begin{remark}
Suppose $V$ is a vertex operator algebra.
Let $g$ be an automorphism of $V$ of finite order $t$.
In \cite{DJ2}, a product $*^{n}_{g,m,p}$ on $V$ and a quotient space $A_{g,n,m}(V)=V/O_{g,n,m}(V)$ of $V$
are constructed for each $n,p,m\in(1/t)\N$.
If $g=\id_{V}$, then $*^{n}_{g,m,p}=*^{n}_{m,p}$ and
$A_{g,n,m}(V)=A_{n,m}(V)$, where $*^{n}_{m,p}$ is a product on $V$ and
$A_{n,m}(V)$ is a quotient space of $V$ constructed in \cite{DJ1}.

We shall discuss a relation between $A_{g,n,m}(V)$ and $\zhumod{T}{n}{m}(V)$.
Suppose $T=1$. Then $*^{1}_{n,p,m}=*^{n}_{m,p}$ by the definition.
Moreover, $\zhuOzero{1}{n}{m}(V)+\zhuOi{1}{n}{m}(V)=O^{\prime}_{n,m}(V)$
by \eqref{eq:def-otnms} and \eqref{eq:def-otnm},
where $O^{\prime}_{n,m}(V)$ is the subspace of $V$ defined on p.\! 801 in \cite{DJ1}. 
Thus, $\zhuO{1}{n}{m}(V)=O_{n,m}(V)$ and $\zhumod{1}{n}{m}(V)=A_{n,m}(V)$.

We shall use the notation in Remark \ref{remark:twisted-vt-Y}
and \cite{DJ2}.
For homogeneous $a,b\in V$ and 
$P(\wz)\in \zhuOi{t}{n}{m}(\wt a,\wt b ; \wz)$, 
the definition of $\zhuOi{t}{n}{m}(\wt a,\wt b ; \wz)$ implies
\begin{align*}
P(\wz)|_{z^j=a_jb}&=\sum_{r=0}^{t-1}P(\wz)|_{z^j=a^{(g,r)}_jb}\in O^{\prime}_{g,n,m}(V),
\end{align*}
where $O^{\prime}_{g,n,m}(V)$ is the subspace of $V$ defined on p.\! 4240 in \cite{DJ2}.
Thus, $\zhuOi{t}{n}{m}(V)$ is a subspace of $O^{\prime}_{g,n,m}(V)$.
We simply write $r=\rbk{p}{n}$, which is defined in \eqref{eq:wrij}.
For $s=0,\ldots,t-1$,
we have
\begin{align*}
&\hunitmu{t}{r}{n}{m}(a^{(g,s)},b,i)-\delta_{r,s}a^{(g,s)}_{i}b\\
&=(\unitmu{t}{r}{n}{m}(\wt a^{(g,s)},\wt b,i;z)-\delta_{r,s}z^i)|_{z^j=a^{(g,s)}_{j}b}\in O^{\prime}_{g,n,m}(V)
\end{align*}
since 
$\unitmu{t}{r}{n}{m}(\wt a^{(g,s)},\wt b,i;z)-\delta_{r,s}z^i\in \zhuOit{t}{s}{n}{m}(\wt a^{(g,s)},\wt b;z)$
by \eqref{eq:unit-qs}.
Therefore, by \eqref{eq:def-mul-x} and \eqref{eq:multi-ab} we have
\begin{align*}
a*^{t}_{n,p,m}b&=\sum_{s\neq r}a^{(g,s)}*^{t}_{n,p,m}b+a^{(g,r)}*^{t}_{n,p,m}\\
&\equiv a^{(g,r)}*^{n}_{g,m,p}b\pmod{O^{\prime}_{g,n,m}(V)}.
\end{align*}
We conclude that $\zhuOi{t}{n}{m}(V)\subset O_{g,n,m}(V)$ and 
$A_{g,n,m}(V)$ is a quotient space of $\zhumod{t}{n}{m}(V)$.

For an automorphism group $G$ of $V$ of finite order,
the same argument as above shows $A_{G,n}(V)$ in \cite{MT} is a quotient space of $\zhualg{|G|}{n}(V)$.
\end{remark}

\section{\label{section:graded-vt}
$(1/T)\N$-graded $(V,T)$-modules and $\zhumod{T}{n}{m}(V)$}

Throughout this section, 
we always assume the properties mentioned at the beginning of Section \ref{section:zhu}
for a vertex algebra $V$ as stated there.
In this section, for $m\in(1/T)\N$ we describe a relation between
the $\zhualg{T}{m}(V)$-modules and the $(1/T)\N$-graded $(V,T)$-modules defined below.
\begin{definition}
{\em A $(1/T)\N$-graded $(V,T)$-module} $M$ is a 
$(V,T)$-module with a $(1/T)\N$-grading $M =\oplus_{n\in(1/T)\N}M(n)$
such that
\begin{align*}
a_{i}M(n)\subset M(n+\wt a-i-1)
\end{align*}
for homogeneous $a\in V$ and $i,n\in (1/T)\N$, where $M(n)=0$ for $n<0$.
\end{definition}
For a $(1/T)\N$-graded $(V,T)$-module $M$, 
a $(V,T)$-submodule $N$ of $M$ is called {\it $(1/T)\N$-graded $(V,T)$-submodule} of $M$
if $N$ is a $(1/T)\N$-graded $(V,T)$-module such that 
every homogeneous subspace of $N$ is contained in some homogeneous subspace of $M$.
A non-zero $(1/T)\N$-graded $(V,T)$-module $M$ is called {\it simple} 
if there is no $(1/T)\N$-graded submodule of $M$ 
except $0$ and $M$ itself.

In the following,
we shall use a similar argument as in \cite[Section 4]{DJ1}.
Throughout this section, $m=l_1+i_1/T, n=l_3+i_3/T\in(1/T)\N$
with $l_1,l_2\in\N$ and $0\leq i_1,i_3\leq T-1$.
Until Proposition \ref{proposition:a-acts-omega},  
$M=\oplus_{i\in(1/T)\N}M(i)$ is a $(1/T)\N$-graded $(V,T)$-module.
Without loss of generality, we can shift the grading of a $(1/T)\N$-graded $(V,T)$-module $M$ so that $M(0)\neq 0$ if $M\neq 0$.

Define a linear map $o_{n,m} : V\rightarrow \Hom_{\C}(M(m),M(n))$ by
\begin{align}
\label{eq:def-onm}
o_{n,m}(a)&=a_{\wt a+m-n-1}
\end{align}
for homogeneous $a\in V$ and extend $o_{n,m}(a)$ for an arbitrary $a\in V$ by linearity.
If $m=n$, we simply write $o=o_{m,m}$.
Define a linear map $Z_{M,n,m}^{(s)}(a,b;-) : 
\C[\wz,\wz^{-1}]\rightarrow \Hom_{\C}(M(m),M)$ by
\begin{align}
\label{eq:zmnm}
Z_{M,n,m}^{(s)}(a,b;\wz^i)&=\cty{s}{M}{a,b}{\wt a+\wt b+m-n-2-i,i}
\end{align}
for $s=0,\ldots,T-1$ and homogeneous $a,b\in V$
and extend $Z_{M,n,m}^{(s)}(a,b;-)$  for arbitrary elements $a,b\in V$ by linearity.
Lemma \ref{lemma:express-ysab} implies that 
the image of $Z_{M,n,m}^{(s)}(a,b;f(z)) : M(m)\rightarrow M$
is contained in $M(n)$ for $f(z)\in \C[\wz,\wz^{-1}]$. That is,
$Z_{M,n,m}^{(s)}(a,b;-) : 
\C[\wz,\wz^{-1}]\rightarrow \Hom_{\C}(M(m),M(n))$.

\begin{lemma}\label{lemma:O-0-zero-x}
For $s=0,\ldots,T-1$ and 
homogeneous $a,b\in V$,
$Z_{M,n,m}^{(s)}(a,b;-)=0$ on $\zhuOit{T}{\ws}{n}{m}(\wt a,\wt b;\wz)$.
\end{lemma}
\begin{proof}
It is sufficient to show that 
$Z_{M,n,m}^{(s)}(a,b;\fo{T}{\ws}{\nthr}{\none}(\wt a,\wt b,d;\wz))=0$ for all $d\in\Z_{\leq 0}$.
Let $w\in M(m)$.
Since $Y_{M}(Y(a,\wx_0)b,\wx_2)w\in M\db{\wx_2^{1/T}}\db{\wx_0}_{\geq -\wt a-\wt b+\lw}$,
it follows by Remark \ref{remark:bound} that
\begin{align}
\label{eq:ab-bound}
\ty{s}{M}{a,b}{\wx_2,\wx_0}(w)\in M\db{\wx_2^{1/T}}\db{\wx_0}_{\geq -\wt a-\wt b+\lw}.
\end{align}

Let
\begin{align*}
j&=\wt a-1+l_1+\delta(\ws\leq i_1)+\frac{\ws}{T},\\
k&=\wt b-1+l_1+\delta(\ws^{\vee}\leq i_1)+\frac{\ws^{\vee}}{T}-d\quad\mbox{ and}\\
l&=-l_1-l_3-\delta(\ws\leq i_1)-\delta(T\leq \ws+i_3)-1+d,
\end{align*}
where $s^{\vee}$ is defined in \eqref{eq:def-vee}.
Since $a_{j+i}=b_{k+i}=0$ on $M(m)$ for all $i\in\N$,
it follows by \eqref{eq:Borcherds-Coeff}, \eqref{eq:ssvee} and \eqref{eq:ab-bound}
that
\begin{align*}
&Z_{M,n,m}^{(s)}(a,b;\fo{T}{\ws}{\nthr}{\none}(\wt a,\wt b,d;\wz))(w)\\
&=\sum_{i=0}^{\wt a+\wt b-1-\lw-l}\binom{j}{i}\cty{s}{M}{a,b}{j+k-i,l+i}(w)
\\
&=\sum_{i=0}^{\infty}\binom{j}{i}\cty{s}{M}{a,b}{j+k-i,l+i}(w)\\
&=0.
\end{align*}
\end{proof}

\begin{lemma}\label{lemma:O-0-zero}
For $u\in \zhuOzero{T}{n}{m}(V)+\zhuOi{T}{n}{m}(V)$,
$o_{n,m}(u)=0$
on $M(m)$.
\end{lemma}
\begin{proof}
Let $a,b$ be homogeneous elements of $V$.
It follows by Lemma \ref{lemma:derivation}
that $o_{n,m}(a_{-2}{\bf 1}+(\wt a+m-n)a)=0$ on $M(m)$.
Let $P(\wz)=\sum_{i\in\Z}\lambda_{i}\wz^i\in \zhuOi{T}{n}{m}(\wt a,\wt b ; \wz)$.
It follows by Lemma \ref{eq:zmnm} that on $M(m)$
\begin{align*}
&o_{n,m}(\sum_{i\in\Z}\lambda_{i}a_ib)
=\sum_{i\in\Z}\lambda_{i}o_{n,m}(a_ib)\\
&=\sum_{s=0}^{T-1}\sum_{i\in\Z}\lambda_{i}\cty{s}{M}{a,b}{\wt a+\wt b+m-n-2-i,i}\\
&=\sum_{s=0}^{T-1}\sum_{i\in\Z}\lambda_{i}Z_{M,n,m}^{(s)}(a,b;\wz^i)\\
&=\sum_{s=0}^{T-1}Z_{M,n,m}^{(s)}(a,b;P(\wz))\\
&=0.
\end{align*}
\end{proof}

\begin{lemma}\label{lemma:hom}
For $a,b\in V$ and $w\in M(m)$
\begin{align*}
o_{n,m}(a*^{T}_{n,p,m}b)w&=o_{n,p}(a)o_{p,m}(b)w.
\end{align*}
\end{lemma}
\begin{proof}
We may assume $a$ and $b$ to be homogeneous elements of $V$.
We simply write $\wr=\rbk{p}{n}$,
which is defined in \eqref{eq:wrij}.
By \eqref{eq:unit-qs} and Lemma \ref{eq:zmnm}, we have
\begin{align}
\label{eq:zmnrab}
&Z_{M,n,m}^{(\wr)}(a,b;\unitmu{T}{\wr}{n}{m}(\wt a,\wt b,j;\wz))(w)\nonumber\\
&=Z_{M,n,m}^{(\wr)}(a,b;z^j)(w)\nonumber\\
&=\cty{\wr}{M}{a,b}{\wt a+\wt b+m-n-2-j,j}
\end{align}
for $j\in\Z$.
We write $\pmul{T}{\nthr}{\ntwo}{\none}(\wt a,\wt b;\wz)=\sum_{i\in\Z}\lambda_{i}\wz^i, \lambda_i\in\C$.
By \eqref{eq:zmnrab}, 
we have
\begin{align}
\label{eq:proof-lemma-hom}
&o_{n,m}(a*^{T}_{n,p,m}b)w
=\sum_{i\in\Z}\lambda_{i}o_{n,m}(a_ib)w\nonumber\\
&=\sum_{s=0}^{T-1}\sum_{i\in\Z}\lambda_{i}\cty{s}{M}{a,b}{\wt a+\wt b+m-n-2-i,i}(w)\nonumber\\
&=\sum_{s=0}^{T-1}Z_{M,n,m}^{(s)}(a,b;\pmul{T}{\nthr}{\ntwo}{\none}(\wt a,\wt b;\wz))(w)\nonumber\\
&=Z_{M,n,m}^{(r)}(a,b;\pmul{T}{\nthr}{\ntwo}{\none}(\wt a,\wt b;\wz))(w)\nonumber\\
&=\sum_{i=0}^{l_2}
\binom{-l_1-l_3+l_2-\delta(\wr\leq i_1)-\delta(T\leq \wr+i_3)}{i}\nonumber\\
&\quad{}\times
\Res_{\wx}\big((1+\wx)^{\wt a-1+l_1+\delta({\wr}\leq i_1)+\wr/T}
\wx^{-l_1-l_3+l_2-\delta(\wr\leq i_1)-\delta(T\leq \wr+i_3)-i}\nonumber\\
&\quad{}\times
\sum_{j\in\Z}
Z_{M,n,m}^{(r)}(a,b;\unitmu{T}{\wr}{\nthr}{\none}(\wt a,\wt b,j;\wz))(w)\wx^{-j-1}\big)\nonumber\\
&=\sum_{i=0}^{l_2}
\binom{-l_1-l_3+l_2-\delta(\wr\leq i_1)-\delta(T\leq \wr+i_3)}{i}\nonumber\\
&\quad{}\times
\Res_{\wx}\big((1+\wx)^{\wt a-1+l_1+\delta({\wr}\leq i_1)+\wr/T}
\wx^{-l_1-l_3+l_2-\delta(\wr\leq i_1)-\delta(T\leq \wr+i_3)-i}
\nonumber\\
&\quad{}\times
\sum_{j\in\Z}
\cty{\wr}{M}{a,b}{\wt a+\wt b+m-n-2-j,j}(w)\wx^{-j-1}\big).
\end{align}
Let $\mu=-l_1-l_3+l_2-\delta(\wr\leq i_1)-\delta(T\leq \wr+i_3)$ and $i\in\Z$.
Then
\begin{align*}
&\Res_{\wx}\big((1+\wx)^{\wt a-1+l_1+\delta({\wr}\leq i_1)+\wr/T}
\wx^{\mu-i}
\nonumber\\
&\quad{}\times
\sum_{j\in\Z}
\cty{\wr}{M}{a,b}{\wt a+\wt b+m-n-2-j,j}(w)\wx^{-j-1}\big)\\
&=
\sum_{k=0}^{\infty}\binom{\wt a-1+l_1+\delta({\wr}\leq i_1)+\wr/T}{k}\\
&\quad{}\times
\cty{\wr}{M}{a,b}{\wt a+\wt b+m-n-2-\mu+i-k,\mu-i+k}(w)\\
&=\sum_{k=0}^{\infty}\binom{\wt a-1+l_1+\delta({\wr}\leq i_1)+\wr/T}{k}\\
&\quad{}\times \Res_{x_2}\Res_{x_1-x_2}\big(x_2^{\wt a+\wt b+m-n-2-\mu+i-k}
(x_1-x_2)^{\mu-i+k}\\
&\quad{}\times \ty{\wr}{M}{a,b}{x_2,x_1-x_2}(w)\big)\\
&=\Res_{x_2}\Res_{x_1-x_2}\big(
x_1^{\wt a-1+l_1+\delta(\wr\leq i_1)+\wr/T}
x_2^{\wt b-1+l_1+\delta(\wr^{\vee}\leq i_1)+r^{\vee}/T-l_2+i-1}\\
&\quad{}\times (x_1-x_2)^{-l_1-l_3+l_2-\delta(\wr\leq i_1)-\delta(T\leq \wr+i_3)-i}
\ty{\wr}{M}{a,b}{x_2,x_1-x_2}(w)\big),
\end{align*}
where we used \eqref{eq:ssvee} in the last step and $\wr^{\vee}$ is defined in \eqref{eq:def-vee}. 
Thus, \eqref{eq:proof-lemma-hom} becomes
\begin{align}
\label{eq:proof-lemma-hom-2}
&o_{n,m}(a*^{T}_{n,p,m}b)w\nonumber\\
&=
\sum_{i=0}^{l_2}
\binom{-l_1-l_3+l_2-\delta(\wr\leq i_1)-\delta(T\leq \wr+i_3)}{i}\nonumber\\
&\quad{}\times\Res_{x_2}\Res_{x_1-x_2}\big(
x_1^{\wt a-1+l_1+\delta(\wr\leq i_1)+\wr/T}
x_2^{\wt b-1+l_1+\delta(\wr^{\vee}\leq i_1)+r^{\vee}/T-l_2+i-1}\nonumber\\
&\quad{}\times (x_1-x_2)^{-l_1-l_3+l_2-\delta(\wr\leq i_1)-\delta(T\leq \wr+i_3)-i} \ty{\wr}{M}{a,b}{x_2,x_1-x_2}(w)\big).
\end{align}
The rest of the proof is the same as that of \cite[Lemma 5.1]{DJ2} by \eqref{eq:residue-yayb}.
\end{proof}

The following result is a direct consequence of Lemma \ref{lemma:hom}.

\begin{corollary}\label{corollary:one-generated}
If $M$ is generated by one homogeneous element $w$ as a $(V,T)$-module,
then $M=\{a_{i}w\ |\ a\in V, i\in (1/T)\Z\}$.
\end{corollary}

We define an $\zhualg{T}{n}(V)$-$\zhualg{T}{m}(V)$-bimodule structure on 
$\Hom_{\C}(M(m),M(n))$
by 
\begin{align*}
(afb)(w)&=a(f(bw))
\end{align*}
for $f\in \Hom_{\C}(M(m),M(n))$, $a\in \zhualg{T}{n}(V),b\in \zhualg{T}{m}(V)$ and $w\in M(m)$.
For a $(V,T)$-module  $W$ and $m\in(1/T)\N$, define
\begin{align*}
&\Omega_{m}(W)\\
&=\{w\in W\ |\ a_{\wt a-1+k}w=0
\mbox{
for all homogeneous $a\in V$ 
and $k>m$
}\}.
\end{align*}
Clearly, $\oplus_{i=0}^{m}M(i)\subset \Omega_{m}(M)$.

Lemmas \ref{lemma:O-0-zero} and \ref{lemma:hom} imply the following results.

\begin{lemma}
For $u\in \zhuO{T}{n}{m}(V)$, $o_{n,m}(u)=0$ on $M(m)$.
The linear map $o_{n,m} : V\rightarrow \Hom_{\C}(M(m),M(n))$ induces an $\zhualg{T}{n}(V)$-$\zhualg{T}{m}(V)$-bimodule homomorphism 
from $\zhumod{T}{n}{m}(V)$ to  $\Hom_{\C}(M(m),M(n))$.
\end{lemma}

\begin{proposition}\label{proposition:a-acts-omega}
Let $W$ be a $(V,T)$-module.
Then $o : V\rightarrow \End_{\C}(\Omega_{m}(W))$ induces a representation of $\zhualg{T}{m}(V)$. 
In particular, $M(m)$ is a left $\zhualg{T}{m}(V)$-module.
\end{proposition}

For a left $A^{T}_{m}(V)$-module $U$,
set
\begin{align*}
M(U)&=\bigoplus_{n\in(1/T)\N}\zhumod{T}{n}{m}(V)\otimes_{\zhualg{T}{m}(V)}U
\end{align*}
and $M(U)(n)=\zhumod{T}{n}{m}(V)\otimes_{\zhualg{T}{m}(V)}U$ for every $n\in (1/T)\N$.
For homogeneous $a\in V$ and $i\in(1/T)\Z$, define an operator $a_{i}$ from $M(U)(n)$ to $M(U)(n+\wt a-i-1)$ by
\begin{align}
\label{eq:def-action-M}
a_i(b\otimes u)&=
\left\{
\begin{array}{ll}
(a*^{T}_{n+\wt a-i-1,n,m}b)\otimes u&\mbox{if }n+\wt a-i-1\geq 0,\\
0&\mbox{if }n+\wt a-i-1<0
\end{array}\right.
\end{align}
for $b\otimes u\in M(U)(n)$ with $b\in V$ and $u\in U$. 
This operation is well-defined (cf. \cite[p.815]{DJ1}).
We extend $a_i$ for an arbitrary $a\in V$ by linearity and set
\begin{align*}
Y_{M(U)}(a,\wx)&=\sum_{i\in(1/T)\Z}a_{i}\wx^{-i-1}:
M(U)\rightarrow M(U)\db{\wx^{1/T}}.
\end{align*}

We shall show $(M(U),Y_{M(U)})$ is a $(1/T)\N$-graded $(V,T)$-module.
For homogeneous $a,b\in V$, $\ws\in\Z$ with $0\leq \ws\leq T-1$, $i\in(1/T)\Z$ and $j\in\Z$,
define a linear map 
$\cty{\ws}{M(U)}{a,b}{i,j} : M(U)(n)\rightarrow M(U)(\wt a+\wt b-i-j-2+n)$
by
\begin{align}
\label{eq:def-ymsabxx}
&\cty{\ws}{M(U)}{a,b}{i,j}(c\otimes u)\nonumber\\
&=
(\hunitmu{T}{\ws}{\wt a+\wt b-i-j-2+n}{n}(a,b,j)*^{T}_{\wt a+\wt b-i-j-2+n,n,m}c)\otimes u
\end{align}
for $c\otimes u\in M(U)(n)$ with $c\in V$ and $u\in U$.
This operation is also well-defined.
We extend $\cty{\ws}{M(U)}{a,b}{i,j}$ 
for arbitrary $a,b\in V$ by linearity
and set
\begin{align*}
\ty{\ws}{M(U)}{a,b}{\wx_2,\wx_0}
&=
\sum_{i\in(1/T)\Z}
\sum_{j\in\Z}
\cty{\ws}{M(U)}{a,b}{i,j}\wx_2^{-i-1}\wx_0^{-j-1}.
\end{align*}
It follows by \eqref{eq:unit-zero} and \eqref{eq:def-unit-ab}
that $\ty{\ws}{M(U)}{a,b}{\wx_2,\wx_0}$ is a linear map from $M(U)$ to $M(U)\db{x_2^{1/T}}\db{x_0}$.

{From} now on, we simply write $M=M(U)$.
By \eqref{eq:sum-unit-s} we have
\begin{align*}
&\sum_{s=0}^{T-1}
Y_{M}^{(\ws)}(a,b;i,j)(c\otimes u)\\
&=\sum_{s=0}^{T-1}(\hunitmu{T}{\ws}{\wt a+\wt b-i-j-2+n}{n}(a,b,j)*^{T}_{\wt a+\wt b-i-j-2+n,n,m}c)\otimes u\\
&=((a_jb)*^{T}_{\wt a+\wt b-i-j-2+n,n,m}c)\otimes u\\
&=(a_jb)_{i}(c\otimes u)
\end{align*}
for homogeneous $a,b\in V$ and $c\otimes u\in M(n)$ with $c\in V$ and $u\in U$.
Thus
\begin{align}
\label{eq:universal-sum}
\sum_{\ws=0}^{T-1}\ty{\ws}{M}{a,b}{x_2,x_0}(w)&=Y_{M}(Y(a,x_0)b,x_2)w
\end{align}
for $w\in M$.

\begin{lemma}\label{lemma:truncated-id}
\begin{enumerate}
\item
$a_{i}M(n)=0$ for homogeneous $a\in V$ and $i>\wt a-1+n$.
\item
$Y_{M}({\bf 1},\wx)=\id_{M}$.
\end{enumerate}
\end{lemma}
\begin{proof}
Clearly, (1) holds.
Let $a\otimes u\in M(n)$ with $a\in V$ and $c\in U$.
By Lemma \ref{lemma:1-poly}, we have
\begin{align*}
{\bf 1}_{i}(a\otimes u)&=({\bf 1}*^{T}_{-i-1+n,n,m}a)\otimes u\\
&=\delta_{i,-1}({\bf 1}_{-1}a)\otimes u
= \delta_{i,-1}a\otimes u
\end{align*}
for $i\in(1/T)\Z$.
\end{proof}

\begin{lemma}\label{lemma:commutative}
Let $a,b$ be homogeneous elements of $V$,
$i,j\in(1/T)\Z$ and
$r$ the integer uniquely determined by the conditions
 $0\leq r\leq T-1$ and $r/T\equiv i\pmod{\Z}$.
Then
\begin{align*}
[a_i,b_j]w&=\sum_{k=0}^{\infty}\binom{i}{k}\cty{r}{M}{a,b}{i+j-k,k}(w)
\end{align*}
for $w\in M$. In particular,
\begin{align*}
(x_1-x_2)^{l}[Y_{M}(a,x_1),Y_{M}(b,x_2)]&=0
\end{align*}
for $l\in\Z_{\geq\max\{\wt a+\wt b-\lw,0\}}$.
\end{lemma}
\begin{proof}
Let $c\otimes u\in M(n)$ with $c\in V$ and $u\in U$. 
By Lemma \ref{lemma:ab-ba-V}, we have
\begin{align*}
&a_{i}b_{j}(c\otimes u)-
b_{j}a_{i}(c\otimes u)\\
&=(a*^{T}_{\wt a+\wt b-i-j-2+n, \wt b-1-j+n,m}(b*^{T}_{\wt b-1-j+n,n,m}c))\otimes u\\
&\quad{}-(b*^{T}_{\wt a+\wt b-i-j-2+n, \wt a-1-i+n,m}(a*^{T}_{\wt a-1-i+n,n,m}c))\otimes u\\
&=((a*^{T}_{\wt a+\wt b-i-j-2+n, \wt b-1-j+n,n}b)*^{T}_{\wt a+\wt b-i-j-2+n,n,m}c)\otimes u\\
&\quad{}-((b*^{T}_{\wt a+\wt b-i-j-2+n, \wt a-1-i+n,n}a)*^{T}_{\wt a+\wt b-i-j-2+n,n,m}c)\otimes u\\
&=(\Res_{\wx}(1+\wx)^{i}
(\sum_{p\in\Z}\hunitmu{T}{\wr}{\wt a+\wt b-i-j-2+n}{n}(a,b,p)x^{-p-1})*^{T}_{\wt a+\wt b-i-j-2+n,n,m}c)\otimes u\\
&=\big(\sum_{k=0}^{\infty}\binom{i}{k}
\hunitmu{T}{\wr}{\wt a+\wt b-i-j-2+n}{n}(a,b,k)*^{T}_{\wt a+\wt b-i-j-2+n,n,m}c\big)\otimes u\\
&=\sum_{k=0}^{\infty}\binom{i}{k}\cty{r}{M}{a,b}{i+j-k,k}(c\otimes u).
\end{align*}
The last formula follows from this and Remark \ref{remark:bound} (cf. \cite[Remark 3.1.13]{LL}).
\end{proof}

We recall that $Y^{\wr}_{M}(a,\wx)$ denotes $\sum_{i\in \wr/T+\Z}a_i\wx^{-i-1}$ for $a\in V$ (cf. \eqref{eq:Ys}).

\begin{lemma}\label{lemma:plus-l}
Let $a,b\in V$ with $a$ being homogeneous,
$l,r\in\N$ with $0\leq r\leq T-1$ and 
$n=l_3+i_3/T\in(1/T)\N$ with $l_3,i_3\in\N$ and $0\leq i_3\leq T-1$.
Then
\begin{align*}
&\Res_{\wx_0}\wx_0^{l}(\wx_2+\wx_0)^{\wt a-1+l_3+\delta(r\leq i_3)+r/T}\ty{r}{M}{a,b}{\wx_2,\wx_0}(w)\\
&
=\Res_{\wx_0}\wx_0^{l}(\wx_0+\wx_2)^{\wt a-1+l_3+\delta(r\leq i_3)+r/T}Y_{M}^{r}(a,\wx_0+\wx_2)Y_{M}(b,\wx_2)w
\end{align*}
for $w\in M(n)$. 
\end{lemma}
\begin{proof}
Using Lemma \ref{lemma:commutative},
we obtain the formula by
the same computation as in the proof of
\cite[Lemma 5.9]{DJ2}.
\end{proof}

\begin{lemma}\label{lemma:minus-l}
Let $a,b\in V$ with $a$ being homogeneous,
$r\in\N$ with $0\leq r\leq T-1$ and 
$n=l_3+i_3/T\in(1/T)\N$ with $l_3\in\N$ and $0\leq i_3\leq T-1$.
Then
\begin{align*}
&\Res_{\wx_0}\wx_0^{-l}(\wx_2+\wx_0)^{\wt a-1+l_3+\delta(r\leq i_3)+r/T}\ty{r}{M}{a,b}{\wx_2,\wx_0}(w)\\
&
=\Res_{\wx_0}\wx_0^{-l}(\wx_0+\wx_2)^{\wt a-1+l_3+\delta(r\leq i_3)+r/T}Y_{M}^{r}(a,\wx_0+\wx_2)Y_{M}(b,\wx_2)w
\end{align*}
for $w\in M(n)$. 
\end{lemma}
\begin{proof}
Let $c\otimes u\in M(n)$ with $c\in V$ and $u\in U$. 
We may assume $b$ to be a homogeneous element of $V$.
We shall show 
\begin{align*}
&\Res_{\wx_0}\wx_0^{-l}(\wx_2+\wx_0)^{\wt a+q}\wx_2^{\wt b-q}\ty{r}{M}{a,b}{\wx_2,\wx_0}(c\otimes u)\\
&
=\Res_{\wx_0}\wx_0^{-l}(\wx_0+\wx_2)^{\wt a+q}\wx_2^{\wt b-q}Y_{M}^{r}(a,\wx_0+\wx_2)Y_{M}(b,\wx_2)c\otimes u,
\end{align*}
where $q=-1+l_3+\delta(r\leq i_3)+r/T$. 
We have
\begin{align}
\label{eq:asso-minus-1}
&\Res_{\wx_0}\wx_0^{-l}(\wx_2+\wx_0)^{\wt a+q}\wx_2^{\wt b-q}\ty{r}{M}{a,b}{\wx_2,\wx_0}(c\otimes u)\nonumber\\
&=\sum_{j=0}^{\infty}\sum_{k\in(1/T)\Z}\binom{\wt a-1+l_3+\delta(r\leq i_3)+r/T}{j}
\wx_2^{-k-1+\wt a+\wt b-j}\nonumber\\
&\quad{}\times \cty{r}{M}{a,b}{k,j-l}(c\otimes u)\nonumber\\
&=\sum_{j=0}^{\infty}\sum_{\begin{subarray}{c}k\in(1/T)\N\end{subarray}}
\binom{\wt a-1+l_3+\delta(r\leq i_3)+r/T}{j}\wx_2^{-l+k-n+1}\nonumber\\
&\quad{}\times \cty{r}{M}{a,b}{\wt a+\wt b-j+l-k+n-2,j-l}(c\otimes u)\nonumber\\
&=\sum_{j=0}^{\infty}\sum_{\begin{subarray}{c}k\in(1/T)\N\end{subarray}}
\binom{\wt a-1+l_3+\delta(r\leq i_3)+r/T}{j}\wx_2^{-l+k-n+1}\nonumber\\
&\quad{}\times (\hunitmu{T}{\wr}{k}{n}(a,b,j-l)*^{T}_{k,n,m}c)\otimes u\nonumber\\
&=\sum_{\begin{subarray}{c}k\in(1/T)\N\end{subarray}}
\wx_2^{-l+k-n+1}(\Res_{\wx}\wx^{-l}(1+\wx)^{\wt a-1+l_3+\delta(r\leq i_3)+r/T}\nonumber\\
&\quad{}\times
(\sum_{j\in\Z} \hunitmu{T}{\wr}{k}{n}(a,b,j)x^{-j-1})*^{T}_{k,n,m}c)\otimes u.
\end{align}

On the other hand, applying the same computation as in the proof of \cite[Lemma 5.10]{DJ2}
to 
\begin{align*}
&\Res_{\wx_0}\wx_0^{-l}(\wx_0+\wx_2)^{\wt a+q}
\wx_2^{\wt b-q}
Y_{M}^{\wr}(a,\wx_0+\wx_2)Y_{M}(b,\wx_2)(c\otimes u)\\
&=\sum_{s=0}^{T-1}
\Res_{\wx_0}\wx_0^{-l}(\wx_0+\wx_2)^{\wt a+q}
\wx_2^{\wt b-q}
Y_{M}^{\wr}(a,\wx_0+\wx_2)Y_{M}^{\ws}(b,\wx_2)(c\otimes u),
&\end{align*}
we have
\begin{align}
\label{eq:asso-minus-2}
&\Res_{\wx_0}\wx_0^{-l}(\wx_0+\wx_2)^{\wt a+q}
\wx_2^{\wt b-q}
Y_{M}^{\wr}(a,\wx_0+\wx_2)Y_{M}(b,\wx_2)(c\otimes u)\nonumber\\
&=\sum_{s=0}^{T-1}\sum_{\begin{subarray}{c}k\in (i_3-r-s)/T+\Z\nonumber\\ 0\leq k\end{subarray}}\wx_2^{-l+k-n+1}\nonumber\\
&\quad{}\times\sum_{\begin{subarray}{c}j\in (i_3-s)/T+\Z\nonumber\\ 0\leq j\leq k+l_3+\delta(r\leq i_3)+r/T-l\end{subarray}}
\binom{-l}{-j+l_3+\delta(r\leq i_3)+r/T-l+k}\nonumber\\
&\quad{}\times(-1)^{-j+l_3+\delta(r\leq i_3)+r/T-l+k}
(a*^{T}_{k,j,m}(b*^{T}_{j,n,m}c))\otimes u\nonumber\\
&\equiv\sum_{s=0}^{T-1}\sum_{\begin{subarray}{c}k\in (i_3-r-s)/T+\Z\nonumber\\ 0\leq k\end{subarray}}\wx_2^{-l+k-n+1}\nonumber\\
&\quad{}\times\sum_{\begin{subarray}{c}j\in (i_3-s)/T+\Z\nonumber\\ 0\leq j\leq k+l_3+\delta(r\leq i_3)+r/T-l\end{subarray}}
\binom{-l}{-j+l_3+\delta(r\leq i_3)+r/T-l+k}\nonumber\\
&\quad{}\times(-1)^{-j+l_3+\delta(r\leq i_3)+r/T-l+k}
((a*^{T}_{k,j,n}b)*^{T}_{k,n,m}c))\otimes u\nonumber\\
&\qquad\pmod{\zhuO{T}{n}{m}(V)\db{\wx_2}}.
\end{align}
Moreover, for each $k=l_4+(i_3-r-s)/T\in (i_3-r-s)/T+\Z$ with $k\geq 0$ and $l_4\in\Z$,
we have
\begin{align}
\label{eq:asso-minus-3}
&\sum_{\begin{subarray}{c}j\in (i_3-s)/T+\Z\\ 0\leq j\leq k+l_3+\delta(r\leq i_3)+r/T-l\end{subarray}}
\binom{-l}{-j+l_3+\delta(r\leq i_3)+r/T-l+k}\nonumber\\
&\quad{}\times(-1)^{-j+l_3+\delta(r\leq i_3)+r/T-l+k}a*^{T}_{k,j,n}b\nonumber\\
&=\sum_{p=0}^{l_4+l_3+\delta(r\leq i_3)+\delta(s\leq i_3)-l-1}
\binom{-l}{p}(-1)^{p}
\sum_{i=0}^{l_4+l_3+\delta(r\leq i_3)+\delta(s\leq i_3)-l-1-p}
\binom{-p-l}{i}\nonumber\\
&\quad\times
\Res_{\wx}(1+\wx)^{\wt a-1+l_3+\delta(r\leq i_3)+r/T}\wx^{-p-l-i}
\sum_{j\in\Z} \hunitmu{T}{\wr}{k}{n}(a,b,j)x^{-j-1}\nonumber\\
&=\Res_{\wx}\wx^{-l}(1+\wx)^{\wt a-1+l_3+\delta(r\leq i_3)+r/T}
\sum_{j\in\Z} \hunitmu{T}{\wr}{k}{n}(a,b,j)x^{-j-1}
\end{align}
by \cite[Proposition 5.3]{DJ1}.
By \eqref{eq:asso-minus-1}--\eqref{eq:asso-minus-3}
the proof is complete.
\end{proof}

By Lemmas \ref{lemma:plus-l} and \ref{lemma:minus-l},
we have the following result.
\begin{lemma}\label{lemma:associative}
Let $a,b\in V$ with $a$ being homogeneous,
$r\in\N$ with $0\leq r\leq T-1$ and 
$n=l_3+i_3/T\in(1/T)\N$ with $l_3,i_3\in\N$ and $0\leq i_3\leq T-1$.
Then
\begin{align*}
&(\wx_2+\wx_0)^{\wt a-1+l_3+\delta(r\leq i_3)+r/T}\ty{r}{M}{a,b}{\wx_2,\wx_0}\\
&
=(\wx_0+\wx_2)^{\wt a-1+l_3+\delta(r\leq i_3)+r/T}Y_{M}^{r}(a,\wx_0+\wx_2)Y_{M}(b,\wx_2)
\end{align*}
on $M(n)$. 
\end{lemma}

By \eqref{eq:universal-sum} and Lemmas \ref{lemma:Borcherds}, \ref{lemma:truncated-id},
\ref{lemma:commutative} and \ref{lemma:associative}, the same argument as in the proof of \cite[Theorem 4.13]{DJ1}
shows the following theorem.
\begin{theorem}\label{theorem:zhu-correspondence}
Let $U$ be a left $\zhualg{T}{m}(V)$-module.
Then $M(U)=\oplus_{n\in(1/T)\N}\zhumod{T}{n}{m}(V)\otimes_{\zhualg{T}{m}(V)}U$
is a $(1/T)\N$-graded $(V,T)$-module with $M(U)(n)=\zhumod{T}{n}{m}(V)\otimes_{\zhualg{T}{m}(V)}U$
and the following universal property:
for a $(V,T)$-module $W$
and an $\zhualg{T}{m}(V)$-homomorphism $\sigma : U\rightarrow \Omega_{m}(W)$,
there is a unique homomorphism $\bar{\sigma} : M(U)\rightarrow W$ of
$(V,T)$-modules that extends $\sigma$.
Moreover, if $U$ cannot factor through $\zhualg{T}{m-1/T}(V)$, then $M(U)(0)\neq 0$.
\end{theorem}

The following result immediately follows from Theorem \ref{theorem:zhu-correspondence} (cf. \cite[Theorem 4.9]{DLM2}).

\begin{corollary}\label{corollary:zhu-correspondence}
For every $m\in(1/T)\N$, there is a bijection between the set of isomorphism classes of 
simple left $\zhualg{T}{m}(V)$-modules which cannot factor through $\zhualg{T}{m-1/T}(V)$
and that of simple $(1/T)\N$-graded $(V,T)$-modules.
\end{corollary}

\section{\label{subsection:appendix}
Appendix}

\subsection{\label{subsection:determinant}
The determinant of a matrix}
In this subsection we shall show that the matrix $\Gamma$ in \eqref{eq:matrix-H} is non-singular.
Let $\wa,t$ be positive integers and $x_0,\ldots,x_{t-1}$ indeterminates.
We denote by $E_n$ the $n\times n$ identity matrix.
Define
$\walpha_{i}^{k}(x_s)=\sum_{j=1}^{k}\binom{x_{\ws}}{i+j}\binom{-x_{\ws}}{k-j}\in\C[x_s]$
for $0\leq \ws\leq t-1, 1\leq k\leq b$ and $i\in\Z$.
Note that
\begin{align}
\label{eq:deg-al}
\deg \walpha_{i}^{k}(x_s)&=i+k
\end{align}
for $i\in\N$.
Define $t$ $\wa t\times \wa $-matrices $A_{s},\ws=0,\ldots,t-1$ by
\begin{align}
\label{eq:matrix-as}
A_{\ws}&=
\begin{pmatrix}
\walpha_{(t-1)\wa -1}^{1}(x_{\ws})&\walpha_{(t-1)\wa -1}^{2}(x_{\ws})&\cdots&\walpha_{(t-1)\wa -1}^{\wa }(x_{\ws})\\
\walpha_{(t-1)\wa -2}^{1}(x_{\ws})&\walpha_{(t-1)\wa -2}^{2}(x_{\ws})&\cdots&\walpha_{(t-1)\wa -2}^{\wa }(x_{\ws})\\
\vdots&\vdots&&\vdots\\
\walpha_{-\wa}^{1}(x_{\ws})&\walpha_{-\wa}^{2}(x_{\ws})&\cdots&\walpha_{-\wa}^{\wa }(x_{\ws})
\end{pmatrix}
\end{align}
and set $A=(A_0\cdots A_{t-1})$.
The following result implies $\Gamma$ is non-singular.
\begin{proposition}
\label{proposition:det-a}
\begin{align*}
\det A&=
\prod_{0\leq i<j\leq t-1}\prod_{k=-\wa +1}^{\wa -1}
\big(\frac{x_i-x_j+k}{\wa (j-i)+k}\big)^{\wa -|k|}.
\end{align*}
\end{proposition}
\begin{proof}
Since 
\begin{align*}
&(\walpha_{i}^{1}(x_{\ws}),\ldots,\alpha_{i}^{\wa }(x_{\ws}))\\
&=
(\binom{x_{\ws}}{i+1},\binom{x_{\ws}}{i+2},\ldots,\binom{x_{\ws}}{i+\wa })
\begin{pmatrix}
1&\binom{-x_{\ws}}{1}&\binom{-x_{\ws}}{2}&\cdots&\binom{-x_{\ws}}{\wa }\\
0&1&\binom{-x_{\ws}}{1}&\ddots&\vdots\\
\vdots&\ddots&\ddots&\ddots&\binom{-x_{\ws}}{2}\\
\vdots&&\ddots&\ddots&\binom{-x_{\ws}}{1}\\
0&\cdots&\cdots&0&1
\end{pmatrix},
\end{align*}
the determinant of $A$ is equal to that of $\wB=(\wB_0\cdots \wB_{t-1})$, where 
\begin{align*}
\wB_{\ws}&=
\begin{pmatrix}
\binom{x_{\ws}}{(t-1)\wa }&\binom{x_{\ws}}{(t-1)\wa +1}&\cdots&\binom{x_{\ws}}{t\wa -1}\\
\binom{x_{\ws}}{(t-1)\wa -1}&\binom{x_{\ws}}{(t-1)\wa }&\cdots&\binom{x_{\ws}}{t\wa -2}\\
\vdots&\vdots&&\vdots\\
\binom{x_{\ws}}{-\wa +1}&\binom{x_{\ws}}{-\wa +2}&\cdots&\binom{x_{\ws}}{0}
\end{pmatrix},\ s=0,\ldots,t-1.
\end{align*}
The same argument as in the proof of \cite[Proposition 9]{MT}
shows that $(x_i-x_j+k)^{\wa -|k|}$ is a factor of $\det B$
for each $0\leq i<j\leq t-1$ and $-\wa +1\leq k\leq \wa -1$.
Thus, there is $c\in \C[x_0,\ldots,x_{t-1}]$ such that
\begin{align*}
\det B&=c\prod_{0\leq i<j\leq t-1}\prod_{k=-\wa +1}^{\wa -1}(x_i-x_j+k)^{\wa -|k|}.
\end{align*}
Since
$\walpha_{i}^{k}(x_s)=\delta_{i+k,0}$ for $i<0$, we have $A_s=\binom{A_s^{\prime}}{E_b}, s=0,\ldots,T-1$,
where
\begin{align*}
\wA_{\ws}^{\prime}&=\begin{pmatrix}
\walpha_{(t-1)\wa -1}^{1}(x_{\ws})&\walpha_{(t-1)\wa -1}^{2}(x_{\ws})&\cdots&\walpha_{(t-1)\wa -1}^{\wa }(x_{\ws})\\
\walpha_{(t-1)\wa -2}^{1}(x_{\ws})&\walpha_{(t-1)\wa -2}^{2}(x_{\ws})&\cdots&\walpha_{(t-1)\wa -2}^{\wa }(x_{\ws})\\
\vdots&\vdots&&\vdots\\
\walpha_{0}^{1}(x_{\ws})&\walpha_{0}^{2}(x_{\ws})&\cdots&\walpha_{0}^{\wa}(x_{\ws})
\end{pmatrix}.
\end{align*}
It follows by
\begin{align}
\label{eq:deform-matrix}
\begin{pmatrix}
E_{(t-1)b}&-A_0^{\prime}\\
O&E_{b}
\end{pmatrix}A&=
\begin{pmatrix}
O&A_1^{\prime}-A_0^{\prime}&\cdots& A_{t-1}^{\prime}-A_0^{\prime}\\
E_{b}&E_{b}&\cdots& E_b
\end{pmatrix}
\end{align}
that $\det A=(-1)^{(t-1)b^2}\det \begin{pmatrix}
A_1^{\prime}-A_0^{\prime}&\cdots& A_{t-1}^{\prime}-A_0^{\prime}\\
\end{pmatrix}$.
Thus,
the degree of $\det A\in \C[x_0,\ldots,x_{t-1}]$ is at most $\binom{t}{2}\wa ^2$ by 
\eqref{eq:deg-al}.
Since the degree of $\prod_{0\leq i<j\leq t-1}\prod_{k=-\wa +1}^{\wa -1}(x_i-x_j+k)^{\wa -|k|}$
is equal to $\binom{t}{2}\wa ^2$,
we have $c\in \C$.

Substituting $((t-1)\wa ,(t-2)\wa ,\ldots,0)$ for 
$(x_0,x_1,\ldots,x_{t-1})$, we obtain
\begin{align*}
1&=c\prod_{0\leq i<j\leq t-1}\prod_{k=-\wa +1}^{\wa -1}(\wa (j-i)+k)^{\wa -|k|}.
\end{align*}
The proof is complete.
\end{proof}

\subsection{\label{subsection:improvement}
Some improvements of results on $A_{G,n}(V)$}

The purpose of this subsection is to improve Theorems 1 and 2 in \cite{MT}. 
Let $V=\oplus_{j=\lw}^{\infty}V_j$ be a vertex operator algebra
and
$G$ an automorphism group of $V$ of finite order $t$.
For $g\in G$ and $n\in (1/t)\N$, $O_{g,n}(V)$ is the subspace of $V$
defined in \cite{DLM3}.

In \cite{MT}, under the condition that $\lw=0$, we constructed an associative algebra
$A_{G,n}(V)$ for each $n\in(1/t)\Z$ in Theorem 1 and
got a duality theorem of Schur-Weyl type in Theorem 2 by
using $A_{G,n}(V)$.
The condition that $\lw=0$ was used in order to show 
the non-singularity of a matrix in \cite[Lemma 3]{MT}.

We shall show \cite[Theorems 1 and 2]{MT} without assuming $\lw=0$. 
To do this, it is sufficient to show the following lemma,
which is an improvement of \cite[Lemma 3]{MT}, by
using $\hunitmu{t}{s}{n}{m}(a,b,i)$ defined in \eqref{eq:def-unit-ab}.
We note that the existence of $\hunitmu{t}{s}{n}{m}(a,b,i)$
follows from Lemma \ref{lemma:iso-oplus} and Proposition \ref{proposition:det-a}.
We use the notation in Remark \ref{remark:twisted-vt-Y}.

\begin{lemma}\label{lemma:app-twisted}
For $a,b\in V=\oplus_{j=\lw}^{\infty}V_j$, $0\leq r\leq t-1$, $p\in\Z$,
$n\in(1/t)\N$ and $g\in G$,
we have
\begin{align*}
\hunitmu{t}{r}{n}{n}(a,b,p)\equiv a^{(g,r)}_{p}b\pmod{O_{g,n}(V)}
\end{align*}
\end{lemma}
\begin{proof}
We may assume $a,b$ to be homogeneous.
We write $n=l+i/t$ with $l,i\in \N$ and $0\leq i\leq t-1$.
We  use the notation in Section \ref{section:subspace}.
It follows from \eqref{eq:res-x-q} that the image of the subspace 
$O(\wt a+\wt b-1-\lw,\wt a-1+l+\delta(s\leq i)+s/t,-2l-3;z)$ of $\C[z,z^{-1}]$ under the map $\C[z,z^{-1}]\ni
f\mapsto f|_{z^j=a^{(g,s)}_jb}\in V$
is contained in $O_{g,n}(V)$ for $s=0,\ldots,t-1$.
By \eqref{eq:unit-s}, we have
\begin{align*}
&\hunitmu{t}{r}{n}{n}(a,b,p)=\unitmu{t}{r}{n}{n}(\wt a,\wt b,p;z)|_{z^j=a_jb}\\
&=\sum_{s\neq r}\unitmu{t}{r}{n}{n}(\wt a,\wt b,p;z)|_{z^j=a^{(g,s)}_jb}\\
&\quad{}+\unitmu{t}{r}{n}{n}(\wt a,\wt b,p;z)|_{z^j=a^{(g,r)}_jb}\\
&\equiv a^{(g,r)}_{p}b\pmod{O_{g,n}(V)}.
\end{align*}
\end{proof}

\section{List of Notations}
\begin{longtable}{lp{10cm}}
$\delta(i\leq j)$ &
$
=
\left\{
\begin{array}{ll}
1& \mbox{if $i\leq j$},\\
0& \mbox{if $i>j$}.
\end{array}\right.$\\
 $Y^{s}_{M}(a,\wx)$ &
$=\sum_{\begin{subarray}{c}i\in s/T+\Z\end{subarray}}
a_{i}\wx^{-i-1}$ where $Y_{M}(a,\wx)=\sum_{\begin{subarray}{c}i\in (1/T)\Z\end{subarray}}
a_{i}\wx^{-i-1}$.\\
$O(N,Q,q;\wz)$ & 
the subspace of $\C[\wz,\wz^{-1}]$
spanned by\\
&$\Res_{\wx}\big((1+\wx)^{Q}\wx^{q+j}\sum_{i\in\Z_{\leq N}}\wz^i\wx^{-i-1}\big)$,
$j=0,-1,\ldots$
and $\wz^i, i\in\Z_{\geq N+1}$ where $N,q\in\Z$ and $Q\in\Q$.\\
$\varphi_{N,\gamma}$&
$\varphi_{N,\gamma}(\wz^i)=
\left\{
\begin{array}{ll}
(-1)^{i+1}
\Res_{\wx}\big((1+\wx)^{\gamma-i}\wx^i\sum_{j\in\Z_{\leq N}}\wz^j\wx^{-j-1}\big)&
\mbox{for $i\leq N$,}\\
\wz^i & \mbox{for $i\geq N+1$.}
\end{array}\right.$ for $z^i\in\C[z,z^{-1}]$.\\
$\rbk{i}{j}$ & the integer uniquely determined by 
the conditions that
$0\leq \rbk{i}{j}\leq T-1\mbox{ and }i-j\equiv \rbk{i}{j}/T\pmod{\Z}$
where $i,j\in(1/T)\Z$ and $T\in\Z_{>0}$.
\\
$\ws^{\vee}$ & the integer uniquely determined by
the conditions that
$
0\leq \ws^{\vee}\leq T-1\mbox{ and }
i_1-i_3\equiv \ws+\ws^{\vee}\pmod T$ where $T\in\Z_{>0}$ and $i_1,i_3,s\in\Z$ with 
$0\leq i_1,i_3,s\leq T$.\\
$\zhuOit{T}{\ws}{n}{m}(\alpha,\beta ; \wz)$&
$=O(\alpha+\beta-1-\lw,\alpha-1+l_1+\delta(\ws\leq i_1)+\dfrac{\ws}{T},$\\
&\qquad\quad $-l_1-l_3-\delta(\ws\leq i_1)-\delta(T\leq \ws+i_3)-1;\wz)$.\\
$\zhuOi{T}{n}{m}(\alpha,\beta ; \wz)$&
$=\bigcap_{s=0}^{T-1}
\zhuOit{T}{\ws}{n}{m}(\alpha,\beta ; \wz)$.\\
$\fo{T}{\ws}{\nthr}{\none}(\alpha,\beta,j;\wz)$&
$=
\Res_{\wx}\big((1+\wx)^{\alpha-1+l_1+\delta(s\leq i_1)+s/T}
\wx^{-l_1-l_3-\delta(s\leq i_1)-\delta(T\leq s+i_3)-1+j}$\\
&$\quad{}\times
\sum\limits_{\begin{subarray}{c}i\in\Z\\i\leq\alpha+\beta-1-\lw\end{subarray}}\wz^{i}\wx^{-i-1}\big)$
\qquad (cf. \eqref{eq:def-o-poly}).\\
$\unitmu{T}{\wr}{\nthr}{\none}(\alpha,\beta,i;\wz)$ & 
the Laurent polynomial in \\
&$\C[z,z^{-1}]_{\alpha+\beta-\lw-T(\alpha+\beta-\lw+l_1+l_3+2),\alpha+\beta-1-\lw}$
uniquely determined by the condition \eqref{eq:unit-s}.\\
$\pmul{T}{\nthr}{\ntwo}{\none}(\alpha,\beta;\wz)$ &
$=\sum_{i=0}^{l_2}
\binom{-l_1-l_3+l_2-\delta(\rbk{p}{n}\leq i_1)-\delta(T\leq \rbk{p}{n}+i_3)}{i}$\\
&$\quad{}\times
\Res_{\wx}\big((1+\wx)^{\alpha-1+l_1+\delta({\rbk{p}{n}}\leq i_1)+\rbk{p}{n}/T}$\\
&$\quad{}\times
\wx^{-l_1-l_3+l_2-\delta(\rbk{p}{n}\leq i_1)-\delta(T\leq \rbk{p}{n}+i_3)-i}
\sum_{j\in\Z}
\unitmu{T}{\rbk{p}{n}}{\nthr}{\none}(\alpha,\beta,j;\wz)\wx^{-j-1}
\big)$
\quad (cf. \eqref{eq:def-mul-x}).
\\
$\hunitmu{T}{\ws}{\nthr}{\none}(a,b,i)$&
$=\unitmu{T}{\ws}{\nthr}{\none}(\wt a,\wt b,i;\wz)|_{\wz^j=a_jb}\in V$.\\
$a*^{T}_{\nthr,\ntwo,\none}b$ & 
$=\pmul{T}{\nthr}{\ntwo}{\none}(\wt a,\wt b;\wz)|_{\wz^j=a_jb}\in V$.\\
$\zhuOzero{T}{n}{m}(V)$
&
the subspace of $V$ spanned by\\
&$\{a_{-2}{\bf 1}+(\wt a+m-n)a\in V\ |\ \mbox{homogeneous }a\in V\}$.\\
$\zhuOi{T}{n}{m}(V)$ &
the subspace of $V$ spanned by\\
&$\Big\{P(\wz)|_{\wz^j=a_jb}\in V\ \Big|\ 
\begin{array}{l}
\mbox{homogeneous $a,b\in V$ and}\\
P(\wz)\in \zhuOi{T}{n}{m}(\wt a,\wt b ; \wz)
\end{array}\Big\}$.\\
$\zhuOii{T}{\nthr}{\none}(V)$ &
the subspace of $V$ spanned by\\
&$u*^{T}_{n,p_3,m}((a*^{T}_{p_3,p_2,p_1}b)*^{T}_{p_3,p_1,m}c-a*^{T}_{p_3,p_2,m}(b*^{T}_{p_2,p_1,m}c))$
for all $a,b,c,u\in V$ and all $p_1,p_2,p_3\in (1/T)\N$.\\
$\zhuOiii{T}{n}{m}(V)$ & 
$=
\sum_{p_1,p_2\in(1/T)\N}(V*^{T}_{n,p_2,p_1}(\zhuOzero{T}{p_2}{p_1}(V)+
\zhuOi{T}{p_2}{p_1}(V))*^{T}_{n,p_1,m}V$.\\
$\zhuO{T}{n}{m}(V)$&
$=
\zhuOzero{T}{n}{m}(V)+
\zhuOi{T}{n}{m}(V)+
\zhuOii{T}{n}{m}(V)+
\zhuOiii{T}{n}{m}(V)$.\\
$Z_{M,n,m}^{(s)}(a,b;\wz^i)$&$=\cty{s}{M}{a,b}{\wt a+\wt b+m-n-2-i,i}$.
\end{longtable}

\end{document}